\documentclass[12pt]{amsart}

\usepackage{amsmath}
\usepackage{amssymb}
\usepackage{enumerate}
\usepackage{graphicx}

\makeatletter
\@namedef{subjclassname@2010}{%
  \textup{2010} Mathematics Subject Classification}
\makeatother

\newtheorem{thm}{Theorem}[section]

\newtheorem{lem}[thm]{Lemma}
\newtheorem{prop}[thm]{Proposition}

\theoremstyle{definition}
\newtheorem{defin}[thm]{Definition}
\newtheorem{rem}[thm]{Remark}

\numberwithin{equation}{section}

\begin{document}


\baselineskip=17pt



\title[Weighted Morrey spaces related to Schr\"odinger operators]{Weighted Morrey spaces related to Schr\"odinger operators with potentials satisfying a reverse H\"older inequality and fractional integrals}

\author[H. Wang]{Hua Wang}
\address{College of Mathematics and Econometrics, Hunan University, Changsha, 410082, P. R. China\\
\&~Department of Mathematics and Statistics, Memorial University, St. John's, NL A1C 5S7, Canada}
\email{wanghua@pku.edu.cn}
\date{}

\begin{abstract}
Let $\mathcal L=-\Delta+V$ be a Schr\"odinger operator on $\mathbb R^d$, $d\geq3$, where $\Delta$ is the Laplacian operator on $\mathbb R^d$ and the nonnegative potential $V$ belongs to the reverse H\"older class $RH_s$ for $s\geq d/2$. For given $0<\alpha<d$, the fractional integrals associated to the Schr\"odinger operator $\mathcal L$ is defined by $\mathcal I_{\alpha}={\mathcal L}^{-{\alpha}/2}$.Suppose that $b$ is a locally integrable function on $\mathbb R^d$, the commutator generated by $b$ and $\mathcal I_{\alpha}$ is defined by $[b,\mathcal I_{\alpha}]f(x)=b(x)\cdot \mathcal I_{\alpha}f(x)-\mathcal I_{\alpha}(bf)(x)$. In this paper, we first introduce some kinds of weighted Morrey spaces related to certain nonnegative potentials belonging to the reverse H\"older class $RH_s$ for $s\geq d/2$. Then we will establish the boundedness properties of the fractional integrals $\mathcal I_{\alpha}$ on these new spaces. Furthermore, weighted strong-type estimate for the corresponding commutator $[b,\mathcal I_{\alpha}]$ in the framework of Morrey spaces is also obtained. The classes of weights, the classes of symbol functions as well as weighted Morrey spaces discussed in this paper are larger than $A_{p,q}$, $\mathrm{BMO}(\mathbb R^d)$ and $L^{p,\kappa}(\mu,\nu)$ corresponding to the classical case (that is $V\equiv0$).
\end{abstract}

\maketitle

\section{Introduction}

\subsection{The critical radius function $\rho(x)$}
Let $d\geq3$ be a positive integer and $\mathbb R^d$ be the $d$-dimensional Euclidean space, and let $V:\mathbb R^d\rightarrow\mathbb R$, $d\geq3$, be a nonnegative locally integrable function that belongs to the \emph{reverse H\"older class} $RH_s$ for some exponent $1<s<\infty$; i.e., there exists a positive constant $C>0$ such that the following reverse H\"older inequality
\begin{equation*}
\left(\frac{1}{|B|}\int_B V(y)^s\,dy\right)^{1/s}\leq C\left(\frac{1}{|B|}\int_B V(y)\,dy\right)
\end{equation*}
holds for every ball $B$ in $\mathbb R^d$. For given $V\in RH_s$ with $s\geq d/2$, we introduce the \emph{critical radius function} $\rho(x)=\rho(x;V)$ which is given by
\begin{equation}\label{rho}
\rho(x):=\sup\bigg\{r>0:\frac{1}{r^{d-2}}\int_{B(x,r)}V(y)\,dy\leq1\bigg\},\quad x\in\mathbb R^d,
\end{equation}
where $B(x,r)$ denotes the open ball centered at $x$ and with radius $r$. It is well known that this auxiliary function satisfies $0<\rho(x)<\infty$ for any $x\in\mathbb R^d$ under the above condition on $V$ (see \cite{shen}). We need the following known result concerning the critical radius function \eqref{rho}.
\begin{lem}[\cite{shen}]\label{N0}
If $V\in RH_s$ with $s\geq d/2$, then there exist two constants $C_0\geq 1$ and $N_0>0$ such that for all $x$ and $y$ in $\mathbb R^d$,
\begin{equation}\label{com}
\frac{\,1\,}{C_0}\left(1+\frac{|x-y|}{\rho(x)}\right)^{-N_0}\leq\frac{\rho(y)}{\rho(x)}\leq C_0\left(1+\frac{|x-y|}{\rho(x)}\right)^{\frac{N_0}{N_0+1}}.
\end{equation}
As a straightforward consequence of \eqref{com}, we can see that for each integer $k\geq1$, the following estimate
\begin{equation}\label{com2}
1+\frac{2^kr}{\rho(y)}\geq \frac{1}{C_0}\left(1+\frac{r}{\rho(x)}\right)^{-\frac{N_0}{N_0+1}}\left(1+\frac{2^kr}{\rho(x)}\right)
\end{equation}
is valid for any $y\in B(x,r)$ with $x\in\mathbb R^d$ and $r>0$, $C_0$ is defined in \eqref{com}.
\end{lem}

\subsection{Fractional integrals associated to Schr\"odinger operators}
Let $V\in RH_s$ for $s\geq d/2$. For such a potential $V$, we consider the \emph{Schr\"odinger differential operator} on $\mathbb R^d$, $d\geq3$,
\begin{equation*}
\mathcal L:=-\Delta+V,
\end{equation*}
and its associated semigroup
\begin{equation*}
\mathcal T_tf(x):=e^{-t\mathcal L}f(x)=\int_{\mathbb R^d}p_t(x,y)f(y)\,dy,\quad t>0,
\end{equation*}
where $p_t(x,y)$ denotes the kernel of the operator $e^{-t\mathcal L},t>0$. From the \emph{Feynman-Kac formula}, it is well-known that
\begin{equation}\label{heat}
\big|p_t(x,y)\big|\leq h_t(x-y),\quad t>0,
\end{equation}
where $h_t$ is the classical heat kernel; i.e.,
\begin{equation*}
h_t(x-y):=(4\pi t)^{-d/2}\exp\Big(-\frac{|x-y|^2}{4t}\Big).
\end{equation*}
Moreover, this estimate \eqref{heat} can be improved when $V$ belongs to the reverse H\"older class $RH_s$ for some $s\geq d/2$ (see \cite{bong1} and \cite{dziu}, for instance). The auxiliary function $\rho(x)$ arises naturally in this context.
\begin{prop}\label{k}
Let $V\in RH_s$ with $s\geq d/2$. For every positive integer $N\geq1$, there exists a positive constant $C_N>0$ such that for all $x$ and $y$ in $\mathbb R^d$,
\begin{equation*}
\big|p_t(x,y)\big|\leq C_N\cdot t^{-d/2}\exp\bigg(-\frac{|x-y|^2}{5t}\bigg)\bigg(1+\frac{\sqrt{t\,}}{\rho(x)}+\frac{\sqrt{t\,}}{\rho(y)}\bigg)^{-N},\quad t>0.
\end{equation*}
\end{prop}
For given $0<\alpha<d$, the \emph{$\mathcal L$-Riesz potential} or \emph{$\mathcal L$-Fractional integral operator} is defined by
\begin{equation}\label{fi}
\begin{split}
\mathcal I_{\alpha}f(x)&:={\mathcal L}^{-{\alpha}/2}f(x)\\
&=\int_0^{\infty}e^{-t\mathcal L}f(x)\,t^{\alpha/2-1}dt.
\end{split}
\end{equation}
In this work we shall be interested in the behavior of the $\mathcal L$-fractional integral operator $\mathcal I_{\alpha}={\mathcal L}^{-{\alpha}/2}$.

\subsection{$A^{\rho,\infty}_p$ and $A^{\rho,\infty}_{p,q}$ weights}
A weight will always mean a nonnegative function which is locally integrable on $\mathbb R^d$. Given a Lebesgue measurable set $E$ and a weight $w$, $|E|$ will denote the Lebesgue measure of $E$ and
\begin{equation*}
w(E)=\int_E w(x)\,dx.
\end{equation*}
Given $B=B(x_0,r)$ and $t>0$, we will write $tB$ for the $t$-dilate ball, which is the ball with the same center $x_0$ and radius $tr$. As in \cite{bong1} (see also \cite{bong2,bong3}), we say that a weight $w$ belongs to the class $A^{\rho,\theta}_p$ for $1<p<\infty$ and $0<\theta<\infty$, if there is a positive constant $C>0$ such that for all balls $B=B(x_0,r)\subset\mathbb R^d$ with $x_0\in\mathbb R^d$ and $r>0$,
\begin{equation*}
\bigg(\frac{1}{|B|}\int_B w(x)\,dx\bigg)^{1/p}\bigg(\frac{1}{|B|}\int_B w(x)^{-{p'}/p}\,dx\bigg)^{1/{p'}}
\leq C\cdot\left(1+\frac{r}{\rho(x_0)}\right)^{\theta},
\end{equation*}
where $p'$ is the dual exponent of $p$ such that $1/p+1/{p'}=1$. For $p=1$ and $0<\theta<\infty$, we also say that a weight $w$ belongs to the class $A^{\rho,\theta}_1$, if there is a positive constant $C>0$ such that for all balls $B=B(x_0,r)\subset\mathbb R^d$ with $x_0\in\mathbb R^d$ and $r>0$,
\begin{equation*}
\frac1{|B|}\int_B w(x)\,dx\leq C\cdot\left(1+\frac{r}{\rho(x_0)}\right)^{\theta}\underset{x\in B}{\mbox{ess\,inf}}\;w(x).
\end{equation*}
For given $1\leq p<\infty$, we define
\begin{equation*}
A^{\rho,\infty}_p:=\bigcup_{\theta>0}A^{\rho,\theta}_p.
\end{equation*}
For any given $\theta>0$, let us introduce the maximal operator which is given in terms of the critical radius function \eqref{rho}.
\begin{equation*}
M_{\rho,\theta}f(x):=\sup_{r>0}\left(1+\frac{r}{\rho(x)}\right)^{-\theta}\frac{1}{|B(x,r)|}\int_{B(x,r)}|f(y)|\,dy,\quad x\in\mathbb R^d.
\end{equation*}
Observe that a weight $w$ belongs to the class $A^{\rho,\infty}_1$ if and only if there exists a positive number $\theta>0$ such that $M_{\rho,\theta}(w)(x)\leq Cw(x)$, for a.e.~$x\in\mathbb R^d$, where the constant $C>0$ is independent of $w$. Since
\begin{equation}\label{cc}
1\leq\left(1+\frac{r}{\rho(x_0)}\right)^{\theta_1}\leq\left(1+\frac{r}{\rho(x_0)}\right)^{\theta_2}
\end{equation}
for $0<\theta_1<\theta_2<\infty$, then for given $p$ with $1\leq p<\infty$, one has
\begin{equation*}
A_p\subset A^{\rho,\theta_1}_p\subset A^{\rho,\theta_2}_p,
\end{equation*}
where $A_p$ denotes the classical Muckenhoupt's class (see \cite[Chapter 7]{grafakos}), and hence $A_p\subset A^{\rho,\infty}_p$. In addition, for some fixed $\theta>0$ (see \cite{tang}),
\begin{equation*}
A^{\rho,\theta}_1\subset A^{\rho,\theta}_{p_1}\subset A^{\rho,\theta}_{p_2},
\end{equation*}
whenever $1\leq p_1<p_2<\infty$. Obviously, for any fixed $\theta>0$,
\begin{equation}\label{2rx}
1\leq\left(1+\frac{2r}{\rho(x_0)}\right)^{\theta}\leq 2^{\theta}\left(1+\frac{r}{\rho(x_0)}\right)^{\theta}.
\end{equation}
To establish weighted norm inequalities for fractional integrals, we need to introduce another weight class $A^{\rho,\infty}_{p,q}$. As in \cite{tang}, we say that a weight $w$ satisfies the condition $A^{\rho,\theta}_{p,q}$ for $1<p<q<\infty$ and $0<\theta<\infty$, if there exists a positive constant $C>0$ such that for any ball $B=B(x_0,r)\subset\mathbb R^d$ with $x_0\in\mathbb R^d$ and $r>0$,
\begin{equation*}
\bigg(\frac{1}{|B|}\int_B w(x)^q\,dx\bigg)^{1/q}\bigg(\frac{1}{|B|}\int_B w(x)^{-{p'}}\,dx\bigg)^{1/{p'}}
\leq C\cdot\left(1+\frac{r}{\rho(x_0)}\right)^{\theta},
\end{equation*}
where $p'=p/{(p-1)}$. We also say that a weight $w$ satisfies the condition $A^{\rho,\theta}_{1,q}$ for $1<q<\infty$ and $0<\theta<\infty$, if there exists a positive constant $C>0$ such that for any ball $B=B(x_0,r)\subset\mathbb R^d$ with $x_0\in\mathbb R^d$ and $r>0$,
\begin{equation*}
\bigg(\frac1{|B|}\int_B w(x)^q\,dx\bigg)^{1/q}\leq C\cdot\left(1+\frac{r}{\rho(x_0)}\right)^{\theta}\underset{x\in B}{\mbox{ess\,inf}}\;w(x).
\end{equation*}
Similarly, for given $p,q$ with $1\leq p<q<\infty$, by \eqref{cc}, one has
\begin{equation*}
A_{p,q}\subset A^{\rho,\theta_1}_{p,q}\subset A^{\rho,\theta_2}_{p,q},
\end{equation*}
whenever $0<\theta_1<\theta_2<\infty$. Here $A_{p,q}$ denotes the classical Muckenhoupt-Wheeden's class (see \cite{muckenhoupt}). We also define
\begin{equation*}
A^{\rho,\infty}_{p,q}:=\bigcup_{\theta>0}A^{\rho,\theta}_{p,q},\qquad 1\leq p<q<\infty.
\end{equation*}
So we have $A_{p,q}\subset A^{\rho,\infty}_{p,q}$. The following results (Lemmas \ref{rh}--\ref{Apq}) are extensions of well-known properties of $A_p$ and $A_{p,q}$ weights. We first present an important property of the classes of weights in $A^{\rho,\theta}_p$ with $1\leq p<\infty$, which was given by Bongioanni, Harboure and Salinas in \cite[Lemma 5]{bong1}.
\begin{lem}[\cite{bong1}]\label{rh}
If $w\in A^{\rho,\theta}_p$ with $0<\theta<\infty$ and $1\leq p<\infty$, then there exist positive constants $\epsilon,\eta>0$ and $C>0$ such that
\begin{equation}\label{rholder}
\left(\frac{1}{|B|}\int_B w(x)^{1+\epsilon}dx\right)^{\frac{1}{1+\epsilon}}
\leq C\left(\frac{1}{|B|}\int_B w(x)\,dx\right)\left(1+\frac{r}{\rho(x_0)}\right)^{\eta}
\end{equation}
for every ball $B=B(x_0,r)$ in $\mathbb R^d$.
\end{lem}

As a direct consequence of Lemma \ref{rh}, we have the following result.
\begin{lem}\label{comparelem}
If $w\in A^{\rho,\theta}_p$ with $0<\theta<\infty$ and $1\leq p<\infty$, then there exist two positive numbers $\delta>0$ and $\eta>0$ such that
\begin{equation}\label{compare}
\frac{w(E)}{w(B)}\leq C\left(\frac{|E|}{|B|}\right)^\delta\left(1+\frac{r}{\rho(x_0)}\right)^{\eta}
\end{equation}
for any measurable subset $E$ of a ball $B=B(x_0,r)$, where $C>0$ is a constant which does not depend on $E$ and $B$.
\end{lem}
\begin{proof}
For any given ball $B=B(x_0,r)$ with $x_0\in\mathbb R^d$ and $r>0$, suppose that $E\subset B$, then by H\"older's inequality with exponent $1+\epsilon$ and \eqref{rholder}, we can deduce that
\begin{equation*}
\begin{split}
w(E)&=\int_{B}\chi_E(x)\cdot w(x)\,dx\\
&\leq
\left(\int_B w(x)^{1+\epsilon}dx\right)^{\frac{1}{1+\epsilon}}
\left(\int_B\chi_E(x)^{\frac{1+\epsilon}{\epsilon}}\,dx\right)^{\frac{\epsilon}{1+\epsilon}}\\
&\leq C|B|^{\frac{1}{1+\epsilon}}\left(\frac{1}{|B|}\int_B w(x)\,dx\right)\left(1+\frac{r}{\rho(x_0)}\right)^{\eta}|E|^{\frac{\epsilon}{1+\epsilon}}\\
&=C\left(\frac{|E|}{|B|}\right)^{\frac{\epsilon}{1+\epsilon}}\left(1+\frac{r}{\rho(x_0)}\right)^{\eta}.
\end{split}
\end{equation*}
This gives \eqref{compare} with $\delta=\epsilon/{(1+\epsilon)}$. Here and in the sequel, the characteristic function of $E$ is denoted by $\chi_E$.
\end{proof}

In view of Lemma \ref{rh}, we now define the reverse H\"older-type class $RH^{\rho,\theta}_q$ that is given in terms of the critical radius function \eqref{rho}. We say that $w\in RH^{\rho,\theta}_q$ for some $1<q<\infty$ and $0<\theta<\infty$, if there exists a positive constant $C>0$ such that the following reverse H\"older-type inequality
\begin{equation*}
\left(\frac{1}{|B|}\int_B w(x)^{q}dx\right)^{1/q}
\leq C\left(\frac{1}{|B|}\int_B w(x)\,dx\right)\left(1+\frac{r}{\rho(x_0)}\right)^{\theta}
\end{equation*}
holds for every ball $B=B(x_0,r)$ in $\mathbb R^d$. The class $RH^{\rho,\infty}_{q}$ is defined as
\begin{equation*}
RH^{\rho,\infty}_{q}:=\bigcup_{\theta>0}RH^{\rho,\theta}_q,\qquad 1<q<\infty.
\end{equation*}
Clearly, one has $RH_q\subset RH^{\rho,\infty}_{q}$.

Let $1<q<\infty$ and $A^q_1=\big\{w:w^q\in A_1\big\}$. It is known that for the classical case (see \cite{johnson}),
\begin{equation*}
A^q_1=A_1\cap RH_q.
\end{equation*}
\begin{lem}\label{relation}
If $1<q<\infty$ and $0<\theta<\infty$, then $w^q\in A^{\rho,\theta}_{1}$ implies that $w\in A^{\rho,\theta/q}_1\bigcap RH^{\rho,\theta/q}_q$.
\end{lem}
\begin{proof}
The conclusion $w\in A^{\rho,\theta/q}_1$ follows easily by H\"older's inequality and the definition of $A^{\rho,\theta}_{1}$. Indeed, for any given ball $B=B(x_0,r)$ with $x_0\in\mathbb R^d$ and $r>0$,
\begin{equation*}
\begin{split}
\frac1{|B|}\int_B w(x)\,dx&\leq\bigg(\frac1{|B|}\int_B w^q(x)\,dx\bigg)^{1/q}\\
&\leq C\cdot\left(1+\frac{r}{\rho(x_0)}\right)^{\theta/q}\Big(\underset{x\in B}{\mbox{ess\,inf}}\;w^q(x)\Big)^{1/q}\\
&=C\cdot\left(1+\frac{r}{\rho(x_0)}\right)^{\theta/q}\underset{x\in B}{\mbox{ess\,inf}}\;w(x).
\end{split}
\end{equation*}
On the other hand, fix a ball $B=B(x_0,r)\subset\mathbb R^d$ and take $y\in B$. Let $E$ be a ball centered at $y_0$ and with radius $h$ which contains $y$. By picking $h$ small enough so that $E\subset B$, then by the condition $w^q\in A^{\rho,\theta}_{1}$, we can deduce that
\begin{equation*}
\begin{split}
\frac{1}{|B|}\int_{B}\chi_E(x)\,dx
&\leq \frac{C}{w^q(B)}\cdot\underset{x\in B}{\mbox{ess\,inf}}\;w^q(x)
\bigg(\int_{B}\chi_E(x)\,dx\bigg)\left(1+\frac{r}{\rho(x_0)}\right)^{\theta}\\
&\leq \frac{C}{w^q(B)}\bigg(\int_{B}\chi_E(x)\cdot w^q(x)\,dx\bigg)\left(1+\frac{r}{\rho(x_0)}\right)^{\theta},
\end{split}
\end{equation*}
which is equivalent to
\begin{equation*}
\frac{w^q(B)}{|B|}\leq C\cdot\frac{w^q(E)}{|E|}\left(1+\frac{r}{\rho(x_0)}\right)^{\theta}.
\end{equation*}
Since this holds for all $E\subset B$ and $y\in E$, then by taking $h\rightarrow 0^+$ and using Lebesgue differentiation theorem,
\begin{equation*}
\frac{w^q(B)}{|B|}\leq C\cdot w^q(y)\left(1+\frac{r}{\rho(x_0)}\right)^{\theta}.
\end{equation*}
Thus, by raising both sides to the power $1/q$ and integrating over $B$, we get
\begin{equation*}
\left(\frac{1}{|B|}\int_B w(y)^{q}dy\right)^{1/q}
\leq C\left(\frac{1}{|B|}\int_B w(y)\,dy\right)\left(1+\frac{r}{\rho(x_0)}\right)^{\theta/q}.
\end{equation*}
This amounts to $w\in RH^{\rho,\theta/q}_q$.
\end{proof}

A subtle interplay between these two classes of weights, $A^{\rho,\infty}_p$ and $A^{\rho,\infty}_{p,q}$, is expressed by the following lemma:
\begin{lem}\label{Apq}
Suppose that $1\leq p<q<\infty$. Then the following statements are true:
\begin{itemize}
  \item if $p>1$ and $0<\theta<\infty$, then $w\in A^{\rho,\theta}_{p,q}$ implies that $w^q\in A^{\rho,\theta\cdot\frac{1}{1/q+1/{p'}}}_t$ with $t=1+q/{p'}$, and $w^{-p'}\in A^{\rho,\theta\cdot\frac{1}{1/q+1/{p'}}}_{t'}$ with $t'=1+{p'}/q$;
  \item if $p=1$ and $0<\theta<\infty$, then $w\in A^{\rho,\theta}_{1,q}$ implies that $w^q\in A^{\rho,\theta\cdot q}_1$.
\end{itemize}
\end{lem}
\begin{proof}
In fact, when $t=1+q/{p'}$, then a simple computation shows that
\begin{equation*}
\frac{\,1\,}{t}=\frac{\,1\,}{q}\cdot\frac{1}{1/q+1/{p'}},\qquad \frac{\,1\,}{t'}=\frac{t-1}{t}=\frac{\,1\,}{p'}\cdot\frac{1}{1/q+1/{p'}},
\end{equation*}
and
\begin{equation*}
q\cdot\Big(-\frac{t'}{\,t\,}\Big)=-q\cdot\frac{1}{t-1}=-p'.
\end{equation*}
If $w\in A^{\rho,\theta}_{p,q}$ with $1<p<q<\infty$ and $0<\theta<\infty$, then we have
\begin{equation*}
\begin{split}
&\bigg(\frac{1}{|B|}\int_B w^q(x)\,dx\bigg)^{1/t}\bigg(\frac{1}{|B|}\int_B w^q(x)^{-{t'}/t}\,dx\bigg)^{1/{t'}}\\
&=\bigg[\bigg(\frac{1}{|B|}\int_B w(x)^q\,dx\bigg)^{1/q}\bigg(\frac{1}{|B|}\int_B w(x)^{-{p'}}\,dx\bigg)^{1/{p'}}\bigg]^{\frac{1}{1/q+1/{p'}}}\\
&\leq C\cdot\left(1+\frac{r}{\rho(x_0)}\right)^{\theta\cdot\frac{1}{1/q+1/{p'}}},
\end{split}
\end{equation*}
which means that $w^q\in A^{\rho,\theta\cdot\frac{1}{1/q+1/{p'}}}_{t}$ with $t=1+q/{p'}$. Here and in the sequel, for any positive number $\gamma>0$, we denote $w^{\gamma}(x)=w(x)^{\gamma}$ by convention. Analogously, it can be easily shown that $w^{-p'}\in A^{\rho,\theta\cdot\frac{1}{1/q+1/{p'}}}_{t'}$ with $t'=1+{p'}/q$. On the other hand, if $w\in A^{\rho,\theta}_{1,q}$ with $1<q<\infty$ and $0<\theta<\infty$, then we have
\begin{equation*}
\begin{split}
\frac1{|B|}\int_B w^q(x)\,dx&\leq C\cdot\left(1+\frac{r}{\rho(x_0)}\right)^{\theta\cdot q}\Big(\underset{x\in B}{\mbox{ess\,inf}}\;w(x)\Big)^q\\
&=C\cdot\left(1+\frac{r}{\rho(x_0)}\right)^{\theta\cdot q}\underset{x\in B}{\mbox{ess\,inf}}\;w^q(x),
\end{split}
\end{equation*}
which means that $w^q\in A^{\rho,\theta\cdot q}_{1}$.
\end{proof}

Given a weight $w$ on $\mathbb R^d$, as usual, the weighted Lebesgue space $L^p(w)$ for $1\leq p<\infty$ is defined to be the set of all functions $f$ such that
\begin{equation*}
\big\|f\big\|_{L^p(w)}:=\bigg(\int_{\mathbb R^d}\big|f(x)\big|^pw(x)\,dx\bigg)^{1/p}<\infty.
\end{equation*}
We also denote by $WL^p(w)$ the weighted weak Lebesgue space consisting of all measurable functions $f$ for which
\begin{equation*}
\big\|f\big\|_{WL^p(w)}:=
\sup_{\lambda>0}\lambda\cdot\Big[w\big(\big\{x\in\mathbb R^d:|f(x)|>\lambda\big\}\big)\Big]^{1/p}<\infty.
\end{equation*}

Note that if $\mathcal L=-\Delta$ is the Laplacian on $\mathbb R^d$, then ${\mathcal L}^{-{\alpha}/2}$ is the classical fractional integral operator $I_{\alpha}=(-\Delta)^{-{\alpha}/2}$ of order $\alpha$. Let $0<\alpha<d$ and $1\leq p<d/{\alpha}$. Define $1<q<\infty$ by the relation $1/q=1/p-{\alpha}/d$. It is well known that when $p>1$ and $w\in A_{p,q}$, $I_{\alpha}$ is bounded from $L^p(w^p)$ into $L^q(w^q)$. When $p=1$ and $w\in A_{1,q}$, $I_{\alpha}$ is bounded from $L^1(w)$ into $WL^q(w^q)$ (see \cite{muckenhoupt}). Recently, Bongioanni et al. \cite[Theorem 4]{bong1} established the weighted boundedness
for fractional integral operators $\mathcal I_{\alpha}$ associated to Schr\"odinger operators defined in \eqref{fi}. They showed that the same estimates also hold for weights in the class $A^{\rho,\infty}_{p,q}$, which is larger than Muckenhoupt-Wheeden's class (another proof was later given by Tang in \cite[Theorem 3.8]{tang}). Their results can be summarized as follows:

\begin{thm}[\cite{bong1}]\label{strong}
Let $0<\alpha<d$, $1<p<d/{\alpha}$, $1/q=1/p-{\alpha}/d$ and $w\in A^{\rho,\infty}_{p,q}$. If $V\in RH_s$ with $s\geq d/2$, then the $\mathcal L$-fractional integral operator $\mathcal I_{\alpha}$ is bounded from $L^p(w^p)$ into $L^q(w^q)$.
\end{thm}

\begin{thm}[\cite{bong1}]\label{weak}
Let $0<\alpha<d$, $p=1$, $q=d/{(d-\alpha)}$ and $w\in A^{\rho,\infty}_{1,q}$. If $V\in RH_s$ with $s\geq d/2$, then the $\mathcal L$-fractional integral operator $\mathcal I_{\alpha}$ is bounded from $L^1(w)$ into $WL^q(w^q)$.
\end{thm}

\subsection{The space $\mathrm{BMO}_{\rho,\infty}(\mathbb R^d)$}
For a locally integrable function $b$ on $\mathbb R^d$ (usually called the \emph{symbol}), we will also consider the commutator operator
\begin{equation}\label{briesz}
[b,\mathcal I_{\alpha}]f(x):=b(x)\cdot \mathcal I_{\alpha}f(x)-\mathcal I_{\alpha}(bf)(x),\quad x\in\mathbb R^d.
\end{equation}
Recently, Bongioanni et al. \cite{bong3} introduced a kind of new spaces $\mathrm{BMO}_{\rho,\infty}(\mathbb R^d)$ defined by
\begin{equation*}
\mathrm{BMO}_{\rho,\infty}(\mathbb R^d):=\bigcup_{\theta>0}\mathrm{BMO}_{\rho,\theta}(\mathbb R^d),
\end{equation*}
where for $0<\theta<\infty$ the space $\mathrm{BMO}_{\rho,\theta}(\mathbb R^d)$ is defined to be the set of all locally integrable functions $b$ satisfying
\begin{equation}\label{BM}
\frac{1}{|B(x_0,r)|}\int_{B(x_0,r)}\big|b(x)-b_{B(x_0,r)}\big|\,dx
\leq C\cdot\left(1+\frac{r}{\rho(x_0)}\right)^{\theta},
\end{equation}
for all balls $B(x_0,r)$ with $x_0\in\mathbb R^d$ and $r>0$, $b_{B(x_0,r)}$ denotes the mean value of $b$ on $B(x_0,r)$; that is,
\begin{equation*}
b_{B(x_0,r)}:=\frac{1}{|B(x_0,r)|}\int_{B(x_0,r)}b(y)\,dy.
\end{equation*}
A norm for $b\in \mathrm{BMO}_{\rho,\theta}(\mathbb R^d)$, denoted by $\|b\|_{\mathrm{BMO}_{\rho,\theta}}$, is given by the infimum of the constants satisfying \eqref{BM}, or equivalently,
\begin{equation*}
\|b\|_{\mathrm{BMO}_{\rho,\theta}}
:=\sup_{B(x_0,r)}\left(1+\frac{r}{\rho(x_0)}\right)^{-\theta}\bigg(\frac{1}{|B(x_0,r)|}\int_{B(x_0,r)}\big|b(x)-b_{B(x_0,r)}\big|\,dx\bigg),
\end{equation*}
where the supremum is taken over all balls $B(x_0,r)$ with $x_0\in\mathbb R^d$ and $r>0$.
With the above definition in mind, one has
\begin{equation*}
\mathrm{BMO}(\mathbb R^d)\subset \mathrm{BMO}_{\rho,\theta_1}(\mathbb R^d)\subset \mathrm{BMO}_{\rho,\theta_2}(\mathbb R^d)
\end{equation*}
for $0<\theta_1<\theta_2<\infty$ by virtue of \eqref{cc}, and hence $\mathrm{BMO}(\mathbb R^d)\subset\mathrm{BMO}_{\rho,\infty}(\mathbb R^d)$. Moreover, the classical BMO space \cite{john} is properly contained in $\mathrm{BMO}_{\rho,\infty}(\mathbb R^d)$ (see \cite{bong2,bong3} for more examples).
We need the following key result for the space $\mathrm{BMO}_{\rho,\theta}(\mathbb R^d)$, which was proved by Tang in \cite[Proposition 4.2]{tang}.
\begin{prop}[\cite{tang}]\label{tangbmo}
Let $b\in\mathrm{BMO}_{\rho,\theta}(\mathbb R^d)$ with $0<\theta<\infty$. Then there exist two positive constants $C_1$ and $C_2$ such that for any given ball $B(x_0,r)$ in $\mathbb R^d$ and for any $\lambda>0$, we have
\begin{equation}\label{tang}
\begin{split}
&\big|\big\{x\in B(x_0,r):|b(x)-b_{B(x_0,r)}|>\lambda\big\}\big|\\
&\leq C_1|B(x_0,r)|
\exp\bigg[-\bigg(1+\frac{r}{\rho(x_0)}\bigg)^{-\theta^{\ast}}\frac{C_2 \lambda}{\|b\|_{\mathrm{BMO}_{\rho,\theta}}}\bigg],
\end{split}
\end{equation}
where $\theta^{\ast}=(N_0+1)\theta$ and $N_0$ is the constant appearing in Lemma \ref{N0}.
\end{prop}
As a consequence of Proposition \ref{tangbmo} and Lemma \ref{comparelem}, we have the following result:
\begin{prop}\label{wangbmo}
Let $b\in\mathrm{BMO}_{\rho,\theta}(\mathbb R^d)$ with $0<\theta<\infty$ and $w\in A^{\rho,\infty}_p$ with $1\leq p<\infty$. Then there exist positive constants $C_1,C_2$ and $\eta>0$ such that for any given ball $B(x_0,r)$ in $\mathbb R^d$ and for any $\lambda>0$, we have
\begin{equation}\label{wang}
\begin{split}
&w\big(\big\{x\in B(x_0,r):|b(x)-b_{B(x_0,r)}|>\lambda\big\}\big)\\
&\leq C_1w\big(B(x_0,r)\big)
\exp\bigg[-\bigg(1+\frac{r}{\rho(x_0)}\bigg)^{-\theta^{\ast}}\frac{C_2 \lambda}{\|b\|_{\mathrm{BMO}_{\rho,\theta}}}\bigg]
\left(1+\frac{r}{\rho(x_0)}\right)^{\eta},
\end{split}
\end{equation}
where $\theta^{\ast}=(N_0+1)\theta$ and $N_0$ is the constant appearing in Lemma \ref{N0}.
\end{prop}

Notice that if $\mathcal L=-\Delta$ is the Laplacian operator on $\mathbb R^d$, then $[b,\mathcal I_{\alpha}]$ is just the commutator $[b,I_{\alpha}]$ of the classical fractional integrals. It is well known that when $b\in \mathrm{BMO}(\mathbb R^d)$, the commutator $[b,I_{\alpha}]$ is bounded from $L^p(w^p)$ ($1<p<d/{\alpha}$) into $L^q(w^q)$ whenever $1/q=1/p-{\alpha}/d$ and $w\in A_{p,q}$.This corresponds to the norm inequalities satisfied by $I_{\alpha}$. In \cite[Theorem 4.4]{tang}, Tang obtained weighted strong-type estimate for the commutator $[b,\mathcal I_{\alpha}]$ of fractional integrals associated to Schr\"odinger operators, when $b$ in a larger space than $\mathrm{BMO}(\mathbb R^d)$, that is the space $\mathrm{BMO}_{\rho,\infty}(\mathbb R^d)$. More precisely, he gave the following weighted result (see also \cite[Theorem 3.5]{bui}).
\begin{thm}[\cite{tang}]\label{cstrong}
Let $0<\alpha<d$, $1<p<d/{\alpha}$, $1/q=1/p-{\alpha}/d$ and $w\in A^{\rho,\infty}_{p,q}$. If $V\in RH_s$ with $s\geq d/2$, then the commutator operator $[b,\mathcal I_{\alpha}]$ is bounded from $L^p(w^p)$ into $L^q(w^q)$, whenever $b$ belongs to $\mathrm{BMO}_{\rho,\infty}(\mathbb R^d)$.
\end{thm}

In this paper, firstly, we will define several classes of weighted Morrey spaces related to certain nonnegative potentials satisfying appropriate reverse H\"older inequality. Secondly, we establish weighted boundedness of fractional integrals $\mathcal I_{\alpha}$ associated to $\mathcal L$ on these new spaces. Finally, we also study the continuity property for the commutators $[b,\mathcal I_{\alpha}]$ of fractional integrals with the new $\mathrm{BMO}$ functions defined above.

Throughout this paper $C$ denotes a positive constant not necessarily the same at each occurrence, and a subscript is added when we wish to make
clear its dependence on the parameter in the subscript. We also use $a\approx b$ to denote the equivalence of $a$ and $b$; that is, there exist two positive constants $C_1$, $C_2$ independent of $a,b$ such that $C_1a\leq b\leq C_2a$.

\section{our main results}
In this section, we introduce some types of weighted Morrey spaces related to the nonnegative potential $V$ and then give our main results.
\begin{defin}
Let $1\leq p<\infty$, $0\leq\kappa<1$ and $\mu,\nu$ be two weights on $\mathbb R^d$. For given $0<\theta<\infty$, the weighted Morrey space $L^{p,\kappa}_{\rho,\theta}(\mu,\nu)$ is defined to be the set of all $p$-locally integrable functions $f$ on $\mathbb R^d$ such that
\begin{equation}\label{morrey1}
\bigg(\frac{1}{\nu(B)^{\kappa}}\int_B\big|f(x)\big|^p\mu(x)\,dx\bigg)^{1/p}
\leq C\cdot\left(1+\frac{r}{\rho(x_0)}\right)^{\theta}
\end{equation}
for every ball $B=B(x_0,r)$ in $\mathbb R^d$. A norm for $f\in L^{p,\kappa}_{\rho,\theta}(\mu,\nu)$, denoted by $\|f\|_{L^{p,\kappa}_{\rho,\theta}(\mu,\nu)}$, is given by the infimum of the constants in \eqref{morrey1}, or equivalently,
\begin{equation*}
\big\|f\big\|_{L^{p,\kappa}_{\rho,\theta}(\mu,\nu)}:=\sup_B\left(1+\frac{r}{\rho(x_0)}\right)^{-\theta}
\bigg(\frac{1}{\nu(B)^{\kappa}}\int_B\big|f(x)\big|^p\mu(x)\,dx\bigg)^{1/p}
<\infty,
\end{equation*}
where the supremum is taken over all balls $B$ in $\mathbb R^d$, $x_0$ and $r$ denote the center and radius of $B$ respectively. Define
\begin{equation*}
L^{p,\kappa}_{\rho,\infty}(\mu,\nu):=\bigcup_{\theta>0}L^{p,\kappa}_{\rho,\theta}(\mu,\nu).
\end{equation*}
If $\mu=\nu=w$, then we denote $L^{p,\kappa}_{\rho,\theta}(w)$ and $L^{p,\kappa}_{\rho,\infty}(w)$ for simplicity.
\end{defin}
Note that this definition does not coincide with the one given in \cite{pan} (see also \cite{tang2} for the unweighted case), but in view of the space $\mathrm{BMO}_{\rho,\infty}(\mathbb R^d)$ and the class $A^{\rho,\infty}_{p,q}$ defined above it is more natural in our setting. Obviously, if we take $\theta=0$ or $V\equiv0$, then this new space is just the familiar weighted Morrey space $L^{p,\kappa}(\mu,\nu)$ (or $L^{p,\kappa}(w)$), which was first introduced by Komori and Shirai in \cite{komori} (see also \cite{wang1}).
\begin{defin}
Let $1\leq p<\infty$, $0\leq\kappa<1$ and $w$ be a weight on $\mathbb R^d$. For given $0<\theta<\infty$, the weighted weak Morrey space $WL^{p,\kappa}_{\rho,\theta}(w)$ is defined to be the set of all measurable functions $f$ on $\mathbb R^d$ such that
\begin{equation*}
\frac{1}{w(B)^{\kappa/p}}\sup_{\lambda>0}\lambda\cdot\Big[w\big(\big\{x\in B:|f(x)|>\lambda\big\}\big)\Big]^{1/p}
\leq C\cdot\left(1+\frac{r}{\rho(x_0)}\right)^{\theta}
\end{equation*}
for every ball $B=B(x_0,r)$ in $\mathbb R^d$, or equivalently,
\begin{equation*}
\big\|f\big\|_{WL^{p,\kappa}_{\rho,\theta}(w)}:=\sup_B\left(1+\frac{r}{\rho(x_0)}\right)^{-\theta}\frac{1}{w(B)^{\kappa/p}}
\sup_{\lambda>0}\lambda\cdot \Big[w\big(\big\{x\in B:|f(x)|>\lambda\big\}\big)\Big]^{1/p}<\infty.
\end{equation*}
Correspondingly, we define
\begin{equation*}
WL^{p,\kappa}_{\rho,\infty}(w):=\bigcup_{\theta>0}WL^{p,\kappa}_{\rho,\theta}(w).
\end{equation*}
\end{defin}
Clearly, if we take $\theta=0$ or $V\equiv0$, then this space is just the weighted weak Morrey space $WL^{p,\kappa}(w)$ (see \cite{wang2}). According to the above definitions, one has
\begin{equation*}
\begin{cases}
L^{p,\kappa}(\mu,\nu)\subset L^{p,\kappa}_{\rho,\theta_1}(\mu,\nu)\subset L^{p,\kappa}_{\rho,\theta_2}(\mu,\nu);&\\
WL^{p,\kappa}(w)\subset WL^{p,\kappa}_{\rho,\theta_1}(w)\subset WL^{p,\kappa}_{\rho,\theta_2}(w).&
\end{cases}
\end{equation*}
for $0<\theta_1<\theta_2<\infty$. Hence $L^{p,\kappa}(\mu,\nu)\subset L^{p,\kappa}_{\rho,\infty}(\mu,\nu)$
(in particular $L^{p,\kappa}(w)\subset L^{p,\kappa}_{\rho,\infty}(w)$) and $WL^{p,\kappa}(w)\subset WL^{p,\kappa}_{\rho,\infty}(w)$ for $(p,\kappa)\in[1,\infty)\times[0,1)$.

The space $L^{p,\kappa}_{\rho,\theta}(w)$ (or $WL^{p,\kappa}_{\rho,\theta}(w)$) could be viewed as an extension of weighted (or weak) Lebesgue space (when $\kappa=\theta=0$). Naturally, one may ask the question whether the above conclusions (i.e., Theorems \ref{strong} and \ref{weak} as well as Theorem \ref{cstrong}) still hold if replacing the weighted Lebesgue spaces by the weighted Morrey spaces. In this work, we give a positive answer to this question. We now state our main results as follows.

For weighted strong-type and weak-type estimates of fractional integrals associated to $\mathcal L$ on $L^{p,\kappa}_{\rho,\infty}(w^p,w^q)$, where $w$ belongs to the new classes of weights, we have
\begin{thm}\label{mainthm:1}
Let $0<\alpha<d$, $1<p<d/{\alpha}$, $1/q=1/p-{\alpha}/d$ and $w\in A^{\rho,\infty}_{p,q}$. If $V\in RH_s$ with $s\geq d/2$ and $0<\kappa<p/q$, then the $\mathcal L$-fractional integral operator $\mathcal I_{\alpha}$ is bounded from $L^{p,\kappa}_{\rho,\infty}(w^p,w^q)$ into $L^{q,{(\kappa q)}/p}_{\rho,\infty}(w^q)$.
\end{thm}

\begin{thm}\label{mainthm:2}
Let $0<\alpha<d$, $p=1$, $q=d/{(d-\alpha)}$ and $w\in A^{\rho,\infty}_{1,q}$. If $V\in RH_s$ with $s\geq d/2$ and $0<\kappa<1/q$, then the $\mathcal L$-fractional integral operator $\mathcal I_{\alpha}$ is bounded from $L^{1,\kappa}_{\rho,\infty}(w,w^q)$ into $WL^{q,(\kappa q)}_{\rho,\infty}(w^q)$.
\end{thm}

Concerning the continuity property of the commutators $[b,\mathcal I_{\alpha}]$ in the setting of weighted Morrey spaces, where $b$ is in the new BMO-type space, we have
\begin{thm}\label{mainthm:3}
Let $0<\alpha<d$, $1<p<d/{\alpha}$, $1/q=1/p-{\alpha}/d$ and $w\in A^{\rho,\infty}_{p,q}$. If $V\in RH_s$ with $s\geq d/2$ and $0<\kappa<p/q$ , then the commutator operator $[b,\mathcal I_{\alpha}]$ is bounded from $L^{p,\kappa}_{\rho,\infty}(w^p,w^q)$ into $L^{q,{(\kappa q)}/p}_{\rho,\infty}(w^q)$, whenever $b\in\mathrm{BMO}_{\rho,\infty}(\mathbb R^d)$.
\end{thm}

Moreover, for the extreme case $\kappa=p/q$ of Theorem \ref{mainthm:1}, we will show that the fractional integrals associated to $\mathcal L$ maps $L^{p,\kappa}_{\rho,\infty}(w^p,w^q)$ continuously into the new space $\mathrm{BMO}_{\rho,\infty}(\mathbb R^d)$. This result may be regarded as a supplement of Theorem \ref{mainthm:1}.
\begin{thm}\label{mainthm:5}
Let $0<\alpha<d$, $1<p<d/{\alpha}$, $1/q=1/p-{\alpha}/d$ and $w\in A^{\rho,\infty}_{p,q}$. If $V\in RH_s$ with $s\geq d/2$ and $\kappa=p/q$, then the $\mathcal L$-fractional integral operator $\mathcal I_{\alpha}$ is bounded from $L^{p,\kappa}_{\rho,\infty}(w^p,w^q)$ into $\mathrm{BMO}_{\rho,\infty}(\mathbb R^d)$.
\end{thm}

\begin{rem}
It is worth pointing out that in the classical case when $V\equiv0$, Theorems \ref{mainthm:1}, \ref{mainthm:2} and \ref{mainthm:3} have been proved by Komori and Shirai in \cite{komori}, while Theorem \ref{mainthm:5} has been shown by the author in \cite{wang3}.
\end{rem}

\section{Proofs of Theorems $\ref{mainthm:1}$ and $\ref{mainthm:2}$}
In this section, we will prove the conclusions of Theorems \ref{mainthm:1} and \ref{mainthm:2}. Let us remind that the $\mathcal L$-Fractional integral operator of order $\alpha\in(0,d)$ can be written as
\begin{equation*}
\begin{split}
\mathcal I_{\alpha}f(x)={\mathcal L}^{-{\alpha}/2}f(x)
&=\int_0^{\infty}e^{-t\mathcal L}f(x)\,t^{\alpha/2-1}dt\\
&=\int_{\mathbb R^d}\bigg[\int_0^{\infty}p_t(x,y)\,t^{\alpha/2-1}dt\bigg]f(y)\,dy.
\end{split}
\end{equation*}
The kernel of the fractional integral operator $\mathcal I_{\alpha}$ will be denoted by $\mathcal K_{\alpha}(x,y)$. Then (see \cite{bong5,bong4})
\begin{equation*}
\mathcal K_{\alpha}(x,y)=\int_0^{\infty}p_t(x,y)\,t^{\alpha/2-1}dt,
\end{equation*}
and
\begin{equation*}
\mathcal I_{\alpha}f(x)=\int_{\mathbb R^d}\mathcal K_{\alpha}(x,y)f(y)\,dy.
\end{equation*}
The following lemma plays a key role in the proof of our main theorems, which can be found in \cite[Proposition 8]{bong5} (see also \cite[Lemma 3.7]{tang}).
\begin{lem}[\cite{bong5}]\label{kernel}
Let $V\in RH_s$ with $s\geq d/2$ and $0<\alpha<d$. For every positive integer $N\geq1$, there exists a positive constant $C_N>0$ such that for all $x$ and $y$ in $\mathbb R^d$,
\begin{equation*}
\big|\mathcal K_{\alpha}(x,y)\big|\leq C_N\bigg(1+\frac{|x-y|}{\rho(x)}\bigg)^{-N}\frac{1}{|x-y|^{d-\alpha}}.
\end{equation*}
\end{lem}

\begin{proof}[Proof of Theorem $\ref{mainthm:1}$]
By definition, we only have to show that for any given ball $B=B(x_0,r)$ of $\mathbb R^d$, there is some $\vartheta>0$ such that
\begin{equation}\label{Main1}
\bigg(\frac{1}{w^q(B)^{\kappa q/p}}\int_B\big|\mathcal I_{\alpha}f(x)\big|^qw^q(x)\,dx\bigg)^{1/q}\leq C\cdot\left(1+\frac{r}{\rho(x_0)}\right)^{\vartheta}
\end{equation}
holds for any $f\in L^{p,\kappa}_{\rho,\infty}(w^p,w^q)$ with $1<p<q<\infty$ and $0<\kappa<p/q$. Suppose that $f\in L^{p,\kappa}_{\rho,\theta}(w^p,w^q)$ for some $\theta>0$ and $w\in A^{\rho,\theta'}_{p,q}$ for some $\theta'>0$. We decompose the function $f$ as
\begin{equation*}
\begin{cases}
f=f_1+f_2\in L^{p,\kappa}_{\rho,\theta}(w^p,w^q);\  &\\
f_1=f\cdot\chi_{2B};\  &\\
f_2=f\cdot\chi_{(2B)^c},
\end{cases}
\end{equation*}
where $2B$ is the ball centered at $x_0$ of radius $2r>0$, and $\chi_{2B}$ is the characteristic function of $2B$. Then by the linearity of $\mathcal I_{\alpha}$, we write
\begin{equation*}
\begin{split}
\bigg(\frac{1}{w^q(B)^{\kappa q/p}}\int_B\big|\mathcal I_{\alpha}f(x)\big|^qw^q(x)\,dx\bigg)^{1/q}
&\leq\bigg(\frac{1}{w^q(B)^{\kappa q/p}}\int_B\big|\mathcal I_{\alpha}f_1(x)\big|^qw^q(x)\,dx\bigg)^{1/q}\\
&+\bigg(\frac{1}{w^q(B)^{\kappa q/p}}\int_B\big|\mathcal I_{\alpha}f_2(x)\big|^qw^q(x)\,dx\bigg)^{1/q}\\
&:=I_1+I_2.
\end{split}
\end{equation*}
We now analyze each term separately. By Theorem \ref{strong}, we get
\begin{equation*}
\begin{split}
I_1&=\bigg(\frac{1}{w^q(B)^{\kappa q/p}}\int_B\big|\mathcal I_{\alpha}f_1(x)\big|^qw^q(x)\,dx\bigg)^{1/q}\\
&\leq C\cdot\frac{1}{w^q(B)^{\kappa/p}}\bigg(\int_{\mathbb R^d}\big|f_1(x)\big|^pw^p(x)\,dx\bigg)^{1/p}\\
&=C\cdot\frac{1}{w^q(B)^{\kappa/p}}\bigg(\int_{2B}\big|f(x)\big|^pw^p(x)\,dx\bigg)^{1/p}\\
&\leq C\big\|f\big\|_{L^{p,\kappa}_{\rho,\theta}(w^p,w^q)}\cdot\frac{w^q(2B)^{\kappa/p}}{w^q(B)^{\kappa/p}}\cdot\left(1+\frac{2r}{\rho(x_0)}\right)^{\theta}.
\end{split}
\end{equation*}
Since $w\in A^{\rho,\theta'}_{p,q}$ with $1<p<q<\infty$ and $0<\theta'<\infty$, then we know that $w^q\in A^{\rho,\theta'\cdot\frac{1}{1/q+1/{p'}}}_t$ with $t=1+q/{p'}$ according to Lemma \ref{Apq}. Now we claim that for every weight $v\in A^{\rho,\tau}_t$ and every ball $B$ in $\mathbb R^d$, there exists a dimensional constant $C>0$ independent of $v$ and $B$ such that
\begin{equation}\label{doubling1}
v\big(2B(x_0,r)\big)\leq C\cdot\left(1+\frac{2r}{\rho(x_0)}\right)^{t\tau}v\big(B(x_0,r)\big).
\end{equation}
In fact, for $1<t<\infty$, by H\"older's inequality and the definition of $A^{\rho,\tau}_t$, we have
\begin{equation*}
\begin{split}
&\frac{1}{|2B|}\int_{2B}\big|\hbar(x)\big|\,dx=\frac{1}{|2B|}\int_{2B}\big|\hbar(x)\big|v(x)^{1/t}v(x)^{-1/t}\,dx\\
&\leq\frac{1}{|2B|}\bigg(\int_{2B}\big|\hbar(x)\big|^tv(x)\,dx\bigg)^{1/t}
\bigg(\int_{2B}v(x)^{-{t'}/t}\,dx\bigg)^{1/{t'}}\\
&\leq \frac{C_v}{v(2B)^{1/t}}\bigg(\int_{2B}\big|\hbar(x)\big|^tv(x)\,dx\bigg)^{1/t}
\left(1+\frac{2r}{\rho(x_0)}\right)^{\tau}.
\end{split}
\end{equation*}
If we take $\hbar(x)=\chi_B(x)$, then the above expression becomes
\begin{equation}\label{W1}
\frac{|B|}{|2B|}\leq C_v\cdot\frac{v(B)^{1/t}}{v(2B)^{1/t}}\left(1+\frac{2r}{\rho(x_0)}\right)^{\tau},
\end{equation}
which in turn implies \eqref{doubling1}. Also observe that
\begin{equation*}
t\theta'\cdot\frac{1}{1/q+1/{p'}}=\big(1+q/{p'}\big)\theta'\cdot\frac{1}{1/q+1/{p'}}=q\theta'.
\end{equation*}
Therefore, in view of \eqref{doubling1} and \eqref{2rx},
\begin{equation*}
\begin{split}
I_1&\leq C_w\big\|f\big\|_{L^{p,\kappa}_{\rho,\theta}(w^p,w^q)}
\cdot\left(1+\frac{2r}{\rho(x_0)}\right)^{(q\theta')\cdot(\kappa/p)}\cdot\left(1+\frac{2r}{\rho(x_0)}\right)^{\theta}\\
&=C_w\big\|f\big\|_{L^{p,\kappa}_{\rho,\theta}(w^p,w^q)}
\cdot\left(1+\frac{2r}{\rho(x_0)}\right)^{\vartheta'}\leq C\cdot\left(1+\frac{r}{\rho(x_0)}\right)^{\vartheta'},
\end{split}
\end{equation*}
where $\vartheta':={(q\theta'\kappa)}/p+\theta$. For the other term $I_2$, notice that for any $x\in B$ and $y\in (2B)^c$, one has $|x-y|\approx|x_0-y|$. It then follows from Lemma \ref{kernel} that for any $x\in B(x_0,r)$ and any positive integer $N$,
\begin{equation}\label{T}
\begin{split}
\big|\mathcal I_{\alpha}f_2(x)\big|&\leq\int_{(2B)^c}|\mathcal K_{\alpha}(x,y)|\cdot|f(y)|\,dy\\
&\leq C_N\int_{(2B)^c}\bigg(1+\frac{|x-y|}{\rho(x)}\bigg)^{-N}\frac{1}{|x-y|^{d-\alpha}}\cdot|f(y)|\,dy\\
&\leq C_{N,d}\int_{(2B)^c}\bigg(1+\frac{|x_0-y|}{\rho(x)}\bigg)^{-N}\frac{1}{|x_0-y|^{d-\alpha}}\cdot|f(y)|\,dy\\
&=C_{N,d}\sum_{k=1}^\infty\int_{2^{k+1}B\backslash 2^k B}\bigg(1+\frac{|x_0-y|}{\rho(x)}\bigg)^{-N}\frac{1}{|x_0-y|^{d-\alpha}}\cdot|f(y)|\,dy\\
&\leq C_{N,d}\sum_{k=1}^\infty\frac{1}{|2^{k+1}B|^{1-(\alpha/d)}}\int_{2^{k+1}B\backslash 2^k B}\bigg(1+\frac{2^kr}{\rho(x)}\bigg)^{-N}|f(y)|\,dy.
\end{split}
\end{equation}
In view of \eqref{com2} in Lemma \ref{N0} and \eqref{2rx}, we further obtain
\begin{align}\label{Tf2}
\big|\mathcal I_{\alpha}f_2(x)\big|
&\leq C\sum_{k=1}^\infty\frac{1}{|2^{k+1}B|^{1-(\alpha/d)}}\int_{2^{k+1}B}\left(1+\frac{r}{\rho(x_0)}\right)^{N\cdot\frac{N_0}{N_0+1}}
\left(1+\frac{2^kr}{\rho(x_0)}\right)^{-N}|f(y)|\,dy\notag\\
&\leq C\sum_{k=1}^\infty\frac{1}{|2^{k+1}B|^{1-(\alpha/d)}}\int_{2^{k+1}B}\left(1+\frac{r}{\rho(x_0)}\right)^{N\cdot\frac{N_0}{N_0+1}}
\left(1+\frac{2^{k+1}r}{\rho(x_0)}\right)^{-N}|f(y)|\,dy.
\end{align}
We consider each term in the sum of \eqref{Tf2} separately. By using H\"older's inequality and $A^{\rho,\theta'}_{p,q}$ condition on $w$, we obtain that for each $k\geq1$,
\begin{equation*}
\begin{split}
&\frac{1}{|2^{k+1}B|^{1-(\alpha/d)}}\int_{2^{k+1}B}\big|f(y)\big|\,dy\\
&\leq\frac{1}{|2^{k+1}B|^{1-(\alpha/d)}}\bigg(\int_{2^{k+1}B}\big|f(y)\big|^pw^p(y)\,dy\bigg)^{1/p}
\bigg(\int_{2^{k+1}B}w(y)^{-{p'}}\,dy\bigg)^{1/{p'}}\\
&\leq C\big\|f\big\|_{L^{p,\kappa}_{\rho,\theta}(w^p,w^q)}\cdot\frac{w^q(2^{k+1}B)^{{\kappa}/p}}{w^q(2^{k+1}B)^{1/q}}
\left(1+\frac{2^{k+1}r}{\rho(x_0)}\right)^{\theta}\left(1+\frac{2^{k+1}r}{\rho(x_0)}\right)^{\theta'}.
\end{split}
\end{equation*}
Hence,
\begin{equation*}
\begin{split}
I_2&\leq C\big\|f\big\|_{L^{p,\kappa}_{\rho,\theta}(w^p,w^q)}\cdot\frac{w^q(B)^{1/q}}{w^q(B)^{{\kappa}/p}}
\sum_{k=1}^\infty\frac{w^q(2^{k+1}B)^{{\kappa}/p}}{w^q(2^{k+1}B)^{1/q}}
\left(1+\frac{r}{\rho(x_0)}\right)^{N\cdot\frac{N_0}{N_0+1}}\left(1+\frac{2^{k+1}r}{\rho(x_0)}\right)^{-N+\theta+\theta'}\\
&=C\big\|f\big\|_{L^{p,\kappa}_{\rho,\theta}(w^p,w^q)}\left(1+\frac{r}{\rho(x_0)}\right)^{N\cdot\frac{N_0}{N_0+1}}
\sum_{k=1}^\infty\frac{w^q(B)^{1/q-\kappa/p}}{w^q(2^{k+1}B)^{1/q-\kappa/p}}
\left(1+\frac{2^{k+1}r}{\rho(x_0)}\right)^{-N+\theta+\theta'}.
\end{split}
\end{equation*}
Recall that $w^q\in A^{\rho,\theta'\cdot\frac{1}{1/q+1/{p'}}}_t$ with $t=1+q/{p'}$ and $0<\theta'<\infty$, then there exist two positive numbers $\delta,\eta>0$ such that \eqref{compare} holds. This allows us to obtain
\begin{equation*}
\begin{split}
I_2&\leq C\big\|f\big\|_{L^{p,\kappa}_{\rho,\theta}(w^p,w^q)}
\left(1+\frac{r}{\rho(x_0)}\right)^{N\cdot\frac{N_0}{N_0+1}}\sum_{k=1}^\infty\left(\frac{|B|}{|2^{k+1}B|}\right)^{\delta{(1/q-\kappa/p)}}\\
&\quad\times\left(1+\frac{2^{k+1}r}{\rho(x_0)}\right)^{\eta{(1/q-\kappa/p)}}\left(1+\frac{2^{k+1}r}{\rho(x_0)}\right)^{-N+\theta+\theta'}\\
&=C\big\|f\big\|_{L^{p,\kappa}_{\rho,\theta}(w^p,w^q)}\left(1+\frac{r}{\rho(x_0)}\right)^{N\cdot\frac{N_0}{N_0+1}}\\
&\quad\sum_{k=1}^\infty\left(\frac{|B|}{|2^{k+1}B|}\right)^{\delta{(1/q-\kappa/p)}}
\left(1+\frac{2^{k+1}r}{\rho(x_0)}\right)^{-N+\theta+\theta'+\eta{(1/q-\kappa/p)}}.
\end{split}
\end{equation*}
Thus, by choosing $N$ large enough so that $N>\theta+\theta'+\eta{(1/q-\kappa/p)}$, and the last series is convergent, we then have
\begin{equation*}
\begin{split}
I_2&\leq C\big\|f\big\|_{L^{p,\kappa}_{\rho,\theta}(w^p,w^q)}
\left(1+\frac{r}{\rho(x_0)}\right)^{N\cdot\frac{N_0}{N_0+1}}\sum_{k=1}^\infty\left(\frac{|B|}{|2^{k+1}B|}\right)^{\delta{(1/q-\kappa/p)}}\\
&\leq C\left(1+\frac{r}{\rho(x_0)}\right)^{N\cdot\frac{N_0}{N_0+1}},
\end{split}
\end{equation*}
where the last inequality follows from the fact that $1/q-\kappa/p>0$. Summing up the above estimates for $I_1$ and $I_2$ and letting $\vartheta=\max\big\{\vartheta',N\cdot\frac{N_0}{N_0+1}\big\}$, we obtain our desired inequality \eqref{Main1}. This completes the proof of Theorem \ref{mainthm:1}.
\end{proof}

\begin{proof}[Proof of Theorem $\ref{mainthm:2}$]
To prove Theorem \ref{mainthm:2}, by definition, it suffices to prove that for each given ball $B=B(x_0,r)$ of $\mathbb R^d$, there is some $\vartheta>0$ such that
\begin{equation}\label{Main2}
\frac{1}{w^q(B)^{\kappa}}\sup_{\lambda>0}\lambda\cdot \Big[w^q\big(\big\{x\in B:|\mathcal I_{\alpha} f(x)|>\lambda\big\}\big)\Big]^{1/q}
\leq C\cdot\left(1+\frac{r}{\rho(x_0)}\right)^{\vartheta}
\end{equation}
holds for any $f\in L^{1,\kappa}_{\rho,\infty}(w,w^q)$ with $1<q<\infty$ and $0<\kappa<1/q$. Now suppose that $f\in L^{1,\kappa}_{\rho,\theta}(w,w^q)$ for some $\theta>0$ and $w\in A^{\rho,\theta'}_{1,q}$ for some $\theta'>0$. We decompose the function $f$ as
\begin{equation*}
\begin{cases}
f=f_1+f_2\in L^{1,\kappa}_{\rho,\theta}(w,w^q);\  &\\
f_1=f\cdot\chi_{2B};\  &\\
f_2=f\cdot\chi_{(2B)^c}.
\end{cases}
\end{equation*}
Then for any given $\lambda>0$, by the linearity of $\mathcal I_{\alpha}$, we can write
\begin{equation*}
\begin{split}
&\frac{1}{w^q(B)^{\kappa}}\lambda\cdot \Big[w^q\big(\big\{x\in B:|\mathcal I_{\alpha}f(x)|>\lambda\big\}\big)\Big]^{1/q}\\
&\leq\frac{1}{w^q(B)^{\kappa}}\lambda\cdot \Big[w^q\big(\big\{x\in B:|\mathcal I_{\alpha}f_1(x)|>\lambda/2\big\}\big)\Big]^{1/q}\\
&+\frac{1}{w^q(B)^{\kappa}}\lambda\cdot \Big[w^q\big(\big\{x\in B:|\mathcal I_{\alpha}f_2(x)|>\lambda/2\big\}\big)\Big]^{1/q}\\
&:=I'_1+I'_2.
\end{split}
\end{equation*}
We first give the estimate for the term $I'_1$. By Theorem \ref{weak}, we get
\begin{equation*}
\begin{split}
I'_1&=\frac{1}{w^q(B)^{\kappa}}\lambda\cdot\Big[w^q\big(\big\{x\in B:|\mathcal I_{\alpha} f_1(x)|>\lambda/2\big\}\big)\Big]^{1/q}\\
&\leq C\cdot\frac{1}{w^q(B)^{\kappa}}\bigg(\int_{\mathbb R^d}\big|f_1(x)\big|w(x)\,dx\bigg)\\
&=C\cdot\frac{1}{w^q(B)^{\kappa}}\bigg(\int_{2B}\big|f(x)\big|w(x)\,dx\bigg)\\
&\leq C\big\|f\big\|_{L^{1,\kappa}_{\rho,\theta}(w,w^q)}\cdot\frac{w^q(2B)^{\kappa}}{w^q(B)^{\kappa}}\cdot\left(1+\frac{2r}{\rho(x_0)}\right)^{\theta}.
\end{split}
\end{equation*}
Since $w\in A^{\rho,\theta'}_{1,q}$ with $1<q<\infty$ and $0<\theta'<\infty$, then we know that $w^q\in A^{\rho,\theta'\cdot q}_1$ according to Lemma \ref{Apq}. We now claim that there exists a dimensional constant $C>0$ independent of $v$ and $B$ such that for every weight $v\in A^{\rho,\tau}_1$ and every ball $B$ in $\mathbb R^d$,
\begin{equation}\label{doubling2}
v\big(2B(x_0,r)\big)\leq C\cdot\left(1+\frac{2r}{\rho(x_0)}\right)^{\tau}v\big(B(x_0,r)\big).
\end{equation}
Similar to the proof of \eqref{doubling1}, by the definition of $A^{\rho,\tau}_1$, we can deduce that
\begin{equation*}
\begin{split}
\frac{1}{|2B|}\int_{2B}\big|\hbar(x)\big|\,dx
&\leq\frac{C_v}{v(2B)}\cdot\underset{x\in 2B}{\mbox{ess\,inf}}\;v(x)
\bigg(\int_{2B}\big|\hbar(x)\big|\,dx\bigg)\left(1+\frac{2r}{\rho(x_0)}\right)^{\tau}\\
&\leq \frac{C_v}{v(2B)}\bigg(\int_{2B}\big|\hbar(x)\big|v(x)\,dx\bigg)\left(1+\frac{2r}{\rho(x_0)}\right)^{\tau}.
\end{split}
\end{equation*}
If we choose $\hbar(x)=\chi_B(x)$, then the above expression becomes
\begin{equation}\label{W2}
\frac{|B|}{|2B|}\leq C_v\cdot\frac{v(B)}{v(2B)}\left(1+\frac{2r}{\rho(x_0)}\right)^{\tau},
\end{equation}
which in turn implies \eqref{doubling2}. Therefore, in view of \eqref{doubling2} and \eqref{2rx},
\begin{equation*}
\begin{split}
I'_1
&\leq C\cdot\left(1+\frac{2r}{\rho(x_0)}\right)^{(\theta'q)\cdot\kappa}\cdot\left(1+\frac{2r}{\rho(x_0)}\right)^{\theta}\\
&=C\cdot\left(1+\frac{2r}{\rho(x_0)}\right)^{\vartheta'}\leq C\cdot\left(1+\frac{r}{\rho(x_0)}\right)^{\vartheta'},
\end{split}
\end{equation*}
where $\vartheta':=\theta'q\kappa+\theta$. As for the second term $I'_2$, by using the pointwise inequality \eqref{Tf2} and Chebyshev's inequality, we deduce that
\begin{equation}\label{Tf2pr}
\begin{split}
I'_2&=\frac{1}{w^q(B)^{\kappa}}\lambda\cdot\Big[w^q\big(\big\{x\in B:|\mathcal I_{\alpha}f_2(x)|>\lambda/2\big\}\big)\Big]^{1/q}\\
&\leq\frac{2}{w^q(B)^{\kappa}}\bigg(\int_{B}\big|\mathcal I_{\alpha}f_2(x)\big|^qw^q(x)\,dx\bigg)^{1/q}\\
&\leq C\cdot\frac{w^q(B)^{1/q}}{w^q(B)^{\kappa}}
\sum_{k=1}^\infty\frac{1}{|2^{k+1}B|^{1-(\alpha/d)}}\int_{2^{k+1}B}\left(1+\frac{r}{\rho(x_0)}\right)^{N\cdot\frac{N_0}{N_0+1}}
\left(1+\frac{2^{k+1}r}{\rho(x_0)}\right)^{-N}|f(y)|\,dy.
\end{split}
\end{equation}
We consider each term in the sum of \eqref{Tf2pr} separately. By Lemma \ref{relation}, we know that $w\in A^{\rho,\theta'}_1\bigcap RH^{\rho,\theta'}_q$ because $w^q$ is in $A^{\rho,\theta'\cdot q}_{1}$. Thus, for each $k\geq1$, we compute
\begin{equation*}
\begin{split}
&\frac{1}{|2^{k+1}B|^{1-(\alpha/d)}}\int_{2^{k+1}B}\big|f(y)\big|\,dy\\
&\leq C_w\cdot\frac{|2^{k+1}B|^{\alpha/d}}{w(2^{k+1}B)}\cdot\underset{y\in 2^{k+1}B}{\mbox{ess\,inf}}\;w(y)
\bigg(\int_{2^{k+1}B}\big|f(y)\big|\,dy\bigg)\left(1+\frac{2^{k+1}r}{\rho(x_0)}\right)^{\theta'}\\
&\leq C_w\cdot\frac{|2^{k+1}B|^{\alpha/d}}{w(2^{k+1}B)}\bigg(\int_{2^{k+1}B}\big|f(y)\big|w(y)\,dy\bigg)\left(1+\frac{2^{k+1}r}{\rho(x_0)}\right)^{\theta'}\\
&\leq C\big\|f\big\|_{L^{1,\kappa}_{\rho,\theta}(w,w^q)}\cdot\frac{|2^{k+1}B|^{\alpha/d}}{w(2^{k+1}B)}w^q\big(2^{k+1}B\big)^{\kappa}
\left(1+\frac{2^{k+1}r}{\rho(x_0)}\right)^{\theta}\left(1+\frac{2^{k+1}r}{\rho(x_0)}\right)^{\theta'}.
\end{split}
\end{equation*}
In addition, by reverse H\"older-type inequality ($w\in RH^{\rho,\theta'}_q$), we can see that
\begin{equation*}
\begin{split}
w^q\big(2^{k+1}B\big)^{1/q}&=\left(\int_{2^{k+1}B} w^{q}(y)\,dy\right)^{1/q}\\
&\leq C\big|2^{k+1}B\big|^{1/q}\left(\frac{1}{|2^{k+1}B|}\int_{2^{k+1}B} w(y)\,dy\right)\left(1+\frac{2^{k+1}r}{\rho(x_0)}\right)^{\theta'}.
\end{split}
\end{equation*}
This indicates that
\begin{equation*}
\frac{|2^{k+1}B|^{\alpha/d}}{w(2^{k+1}B)}\leq C\cdot\frac{1}{w^q(2^{k+1}B)^{1/q}}\left(1+\frac{2^{k+1}r}{\rho(x_0)}\right)^{\theta'}.
\end{equation*}
Consequently,
\begin{equation*}
\begin{split}
I'_2&\leq C\big\|f\big\|_{L^{1,\kappa}_{\rho,\theta}(w,w^q)}
\cdot\frac{w^q(B)^{1/q}}{w^q(B)^{\kappa}}\sum_{k=1}^\infty\frac{w^q(2^{k+1}B)^{\kappa}}{w^q(2^{k+1}B)^{1/q}}
\left(1+\frac{r}{\rho(x_0)}\right)^{N\cdot\frac{N_0}{N_0+1}}\left(1+\frac{2^{k+1}r}{\rho(x_0)}\right)^{-N+\theta+2\theta'}\\
&=C\big\|f\big\|_{L^{1,\kappa}_{\rho,\theta}(w,w^q)}
\left(1+\frac{r}{\rho(x_0)}\right)^{N\cdot\frac{N_0}{N_0+1}}\sum_{k=1}^\infty\frac{w^q(B)^{{1/q-\kappa}}}{w^q(2^{k+1}B)^{{1/q-\kappa}}}
\left(1+\frac{2^{k+1}r}{\rho(x_0)}\right)^{-N+\theta+2\theta'}.
\end{split}
\end{equation*}
Recall that $w^q\in A^{\rho,\theta'\cdot q}_1$ with $0<\theta'<\infty$ and $1<q<\infty$, then there exist two positive numbers $\delta',\eta'>0$ such that \eqref{compare} holds. Therefore,
\begin{equation*}
\begin{split}
I'_2&\leq C\big\|f\big\|_{L^{1,\kappa}_{\rho,\theta}(w,w^q)}
\left(1+\frac{r}{\rho(x_0)}\right)^{N\cdot\frac{N_0}{N_0+1}}\sum_{k=1}^\infty\left(\frac{|B|}{|2^{k+1}B|}\right)^{\delta'{(1/q-\kappa)}}\\
&\quad\times\left(1+\frac{2^{k+1}r}{\rho(x_0)}\right)^{\eta'{(1/q-\kappa)}}\left(1+\frac{2^{k+1}r}{\rho(x_0)}\right)^{-N+\theta+2\theta'}\\
&=C\big\|f\big\|_{L^{1,\kappa}_{\rho,\theta}(w,w^q)}
\left(1+\frac{r}{\rho(x_0)}\right)^{N\cdot\frac{N_0}{N_0+1}}\\
&\quad\times\sum_{k=1}^\infty\left(\frac{|B|}{|2^{k+1}B|}\right)^{\delta'{(1/q-\kappa)}}
\left(1+\frac{2^{k+1}r}{\rho(x_0)}\right)^{-N+\theta+2\theta'+\eta'{(1/q-\kappa)}}.
\end{split}
\end{equation*}
By selecting $N$ large enough so that $N>\theta+2\theta'+\eta'{(1/q-\kappa)}$, we thus have
\begin{equation*}
\begin{split}
I'_2&\leq C\left(1+\frac{r}{\rho(x_0)}\right)^{N\cdot\frac{N_0}{N_0+1}}\sum_{k=1}^\infty\left(\frac{|B|}{|2^{k+1}B|}\right)^{\delta'{(1/q-\kappa)}}\\
&\leq C\left(1+\frac{r}{\rho(x_0)}\right)^{N\cdot\frac{N_0}{N_0+1}},
\end{split}
\end{equation*}
where the last inequality is due to $0<\kappa<1/q$. Let $\vartheta=\max\big\{\vartheta',N\cdot\frac{N_0}{N_0+1}\big\}$. Here $N$ is an appropriate constant. Summing up the above estimates for $I'_1$ and $I'_2$, and then taking the supremum over all $\lambda>0$, we obtain the desired inequality \eqref{Main2}. This finishes the proof of Theorem \ref{mainthm:2}.
\end{proof}

\section{Proof of Theorem $\ref{mainthm:3}$}

For the result involving commutators of fractional integrals associated to Schr\"odinger operators, we need the following properties of $\mathrm{BMO}_{\rho,\infty}(\mathbb R^d)$ functions, which are extensions of well-known properties of $\mathrm{BMO}(\mathbb R^d)$ functions.
\begin{lem}\label{BMO1}
If $b\in \mathrm{BMO}_{\rho,\infty}(\mathbb R^d)$ and $w\in A^{\rho,\infty}_p$ with $1\leq p<\infty$, then there exist positive constants $C>0$ and $\mu>0$ such that for every ball $B=B(x_0,r)$ in $\mathbb R^d$, we have
\begin{equation}\label{wang3}
\bigg(\int_B\big|b(x)-b_B\big|^pw(x)\,dx\bigg)^{1/p}\leq C\cdot w(B)^{1/p}\left(1+\frac{r}{\rho(x_0)}\right)^{\mu},
\end{equation}
where $b_B=\frac{1}{|B|}\int_B b(y)\,dy$.
\end{lem}
\begin{proof}
We may assume that $b\in \mathrm{BMO}_{\rho,\theta}(\mathbb R^d)$ with $0<\theta<\infty$. According to Proposition \ref{wangbmo}, we can deduce that
\begin{equation*}
\begin{split}
&\bigg(\int_B\big|b(x)-b_B\big|^pw(x)\,dx\bigg)^{1/p}\\
&=\bigg(\int_0^{\infty}p\lambda^{p-1}w\big(\big\{x\in B:|b(x)-b_{B}|>\lambda\big\}\big)d\lambda\bigg)^{1/p}\\
&\leq C_1^{1/p}\cdot w(B)^{1/p}\bigg\{\int_0^{\infty}p\lambda^{p-1}\exp\bigg[-\bigg(1+\frac{r}{\rho(x_0)}\bigg)^{-\theta^{\ast}}\frac{C_2 \lambda}{\|b\|_{\mathrm{BMO}_{\rho,\theta}}}\bigg]
\left(1+\frac{r}{\rho(x_0)}\right)^{\eta}d\lambda\bigg\}^{1/p}\\
&=C_1^{1/p}\cdot w(B)^{1/p}
\bigg\{\int_0^{\infty}p\lambda^{p-1}\exp\bigg[-\bigg(1+\frac{r}{\rho(x_0)}\bigg)^{-\theta^{\ast}}\frac{C_2 \lambda}{\|b\|_{\mathrm{BMO}_{\rho,\theta}}}\bigg]d\lambda\bigg\}^{1/p}
\left(1+\frac{r}{\rho(x_0)}\right)^{\eta/p}.
\end{split}
\end{equation*}
Making change of variables, then we get
\begin{equation*}
\begin{split}
&\bigg(\int_B\big|b(x)-b_B\big|^pw(x)\,dx\bigg)^{1/p}\\
&\leq C_1^{1/p}\cdot w(B)^{1/p}\bigg(\int_0^{\infty}ps^{p-1}e^{-s}ds\bigg)^{1/p}
\left(\frac{\|b\|_{\mathrm{BMO}_{\rho,\theta}}}{C_2}\right)\left(1+\frac{r}{\rho(x_0)}\right)^{\theta^{\ast}}\left(1+\frac{r}{\rho(x_0)}\right)^{\eta/p}\\
&=\big[C_1p\Gamma(p)\big]^{1/p}\left(\frac{\|b\|_{\mathrm{BMO}_{\rho,\theta}}}{C_2}\right)\cdot w(B)^{1/p}
\left(1+\frac{r}{\rho(x_0)}\right)^{\theta^{\ast}+\eta/p},
\end{split}
\end{equation*}
which yields the desired inequality \eqref{wang3}, if we choose $C=\big[C_1p\Gamma(p)\big]^{1/p}C_2^{-1}\|b\|_{\mathrm{BMO}_{\rho,\theta}}$ and $\mu=\theta^{\ast}+\eta/p$.
\end{proof}

\begin{lem}\label{BMO2}
If $b\in \mathrm{BMO}_{\rho,\theta}(\mathbb R^d)$ with $0<\theta<\infty$, then for any positive integer $k$, there exists a positive constant $C>0$ such that for every ball $B=B(x_0,r)$ in $\mathbb R^d$, we have
\begin{equation*}
\big|b_{2^{k+1}B}-b_B\big|\leq C\cdot(k+1)\left(1+\frac{2^{k+1}r}{\rho(x_0)}\right)^{\theta}.
\end{equation*}
\end{lem}
\begin{proof}
For any positive integer $k$, we have
\begin{equation*}
\begin{split}
\big|b_{2^{k+1}B}-b_B\big|&\leq \sum_{j=1}^{k+1}\big|b_{2^{j}B}-b_{2^{j-1}B}\big|\\
&=\sum_{j=1}^{k+1}\bigg|\frac{1}{|2^{j-1}B|}\int_{2^{j-1}B}\big[b_{2^{j}B}-b(y)\big]dy\bigg|\\
&\leq\sum_{j=1}^{k+1}\frac{2^d}{|2^{j}B|}\int_{2^{j}B}\big|b(y)-b_{2^{j}B}\big|dy\\
&\leq C_{b,d}\|b\|_{\mathrm{BMO}_{\rho,\theta}}\sum_{j=1}^{k+1}\left(1+\frac{2^jr}{\rho(x_0)}\right)^{\theta}.
\end{split}
\end{equation*}
Since for any $1\leq j\leq k+1$, the following estimate
\begin{equation*}
\left(1+\frac{2^jr}{\rho(x_0)}\right)^{\theta}\leq\left(1+\frac{2^{k+1}r}{\rho(x_0)}\right)^{\theta}
\end{equation*}
holds trivially, and hence
\begin{equation*}
\big|b_{2^{k+1}B}-b_B\big|\leq C\sum_{j=1}^{k+1}\left(1+\frac{2^{k+1}r}{\rho(x_0)}\right)^{\theta}=C\cdot(k+1)\left(1+\frac{2^{k+1}r}{\rho(x_0)}\right)^{\theta}.
\end{equation*}
This is just our desired conclusion.
\end{proof}

Now, we are ready to prove the main result of this section.

\begin{proof}[Proof of Theorem $\ref{mainthm:3}$]
By definition, we only need to prove that for an arbitrary ball $B=B(x_0,r)$ of $\mathbb R^d$ and $0<\alpha<d$, there is some $\vartheta>0$ such that
\begin{equation}\label{Main3}
\bigg(\frac{1}{w^q(B)^{\kappa q/p}}\int_B\big|[b,\mathcal I_{\alpha}]f(x)\big|^qw^q(x)\,dx\bigg)^{1/q}\leq C\cdot\left(1+\frac{r}{\rho(x_0)}\right)^{\vartheta}
\end{equation}
holds for any $f\in L^{p,\kappa}_{\rho,\infty}(w^p,w^q)$ with $1<p<q<\infty$ and $0<\kappa<p/q$, whenever $b$ belongs to $\mathrm{BMO}_{\rho,\infty}(\mathbb R^d)$. Suppose that $f\in L^{p,\kappa}_{\rho,\theta}(w^p,w^q)$ for some $\theta>0$, $w\in A^{\rho,\theta'}_{p,q}$ for some $\theta'>0$ as well as $b\in \mathrm{BMO}_{\rho,\theta''}(\mathbb R^d)$ for some $\theta''>0$. As before, we decompose the function $f$ as
\begin{equation*}
\begin{cases}
f=f_1+f_2\in L^{p,\kappa}_{\rho,\theta}(w^p,w^q);\  &\\
f_1=f\cdot\chi_{2B};\  &\\
f_2=f\cdot\chi_{(2B)^c}.
\end{cases}
\end{equation*}
Then by the linearity of $[b,\mathcal I_{\alpha}]$, we write
\begin{equation*}
\begin{split}
\bigg(\frac{1}{w^q(B)^{\kappa q/p}}\int_B\big|[b,\mathcal I_{\alpha}]f(x)\big|^qw^q(x)\,dx\bigg)^{1/q}
&\leq\bigg(\frac{1}{w^q(B)^{\kappa q/p}}\int_B\big|[b,\mathcal I_{\alpha}]f_1(x)\big|^qw^q(x)\,dx\bigg)^{1/q}\\
&+\bigg(\frac{1}{w^q(B)^{\kappa q/p}}\int_B\big|[b,\mathcal I_{\alpha}]f_2(x)\big|^qw^q(x)\,dx\bigg)^{1/q}\\
&:=J_1+J_2.
\end{split}
\end{equation*}
Now we give the estimates for $J_1$, $J_2$, respectively. According to Theorem \ref{cstrong}, we have
\begin{equation*}
\begin{split}
J_1&\leq C\cdot\frac{1}{w^q(B)^{\kappa/p}}\bigg(\int_{\mathbb R^d}\big|f_1(x)\big|^pw^p(x)\,dx\bigg)^{1/p}\\
&=C\cdot\frac{1}{w^q(B)^{\kappa/p}}\bigg(\int_{2B}\big|f(x)\big|^pw^p(x)\,dx\bigg)^{1/p}\\
&\leq C\big\|f\big\|_{L^{p,\kappa}_{\rho,\theta}(w^p,w^q)}
\cdot\frac{w^q(2B)^{\kappa/p}}{w^q(B)^{\kappa/p}}\cdot\left(1+\frac{2r}{\rho(x_0)}\right)^{\theta}.
\end{split}
\end{equation*}
Moreover, notice that $w^q\in A^{\rho,\theta'\cdot\frac{1}{1/q+1/{p'}}}_t$ with $t=1+q/{p'}$ by Lemma \ref{Apq}, and that
\begin{equation*}
t\theta'\cdot\frac{1}{1/q+1/{p'}}=\big(1+q/{p'}\big)\theta'\cdot\frac{1}{1/q+1/{p'}}=q\theta'.
\end{equation*}
Thus, in view of the inequalities \eqref{doubling1} and \eqref{2rx}, we get
\begin{equation*}
\begin{split}
J_1&\leq C\big\|f\big\|_{L^{p,\kappa}_{\rho,\theta}(w^p,w^q)}
\cdot\left(1+\frac{2r}{\rho(x_0)}\right)^{(q\theta')\cdot(\kappa/p)}\cdot\left(1+\frac{2r}{\rho(x_0)}\right)^{\theta}\\
&=C\big\|f\big\|_{L^{p,\kappa}_{\rho,\theta}(w^p,w^q)}\left(1+\frac{2r}{\rho(x_0)}\right)^{\vartheta'}
\leq C\cdot\left(1+\frac{r}{\rho(x_0)}\right)^{\vartheta'},
\end{split}
\end{equation*}
where $\vartheta':={(q\theta'\kappa)}/p+\theta$. On the other hand, by the definition \eqref{briesz}, we can see that for any $x\in B(x_0,r)$,
\begin{equation*}
\big|[b,\mathcal I_{\alpha}]f_2(x)\big|\leq\int_{\mathbb R^d}\big|b(x)-b(y)\big|\big|\mathcal K_{\alpha}(x,y)f_2(y)\big|\,dy.
\end{equation*}
Adding and subtracting $b_B$ inside the integral we write
\begin{equation}\label{twoterm}
\begin{split}
\big|[b,\mathcal I_{\alpha}]f_2(x)\big|
&\leq\big|b(x)-b_B\big|\int_{\mathbb R^d}\big|\mathcal K_{\alpha}(x,y)f_2(y)\big|\,dy+\int_{\mathbb R^d}\big|b(y)-b_B\big|\big|\mathcal K_{\alpha}(x,y)f_2(y)\big|\,dy\\
&:=\xi(x)+\zeta(x).
\end{split}
\end{equation}
So we can divide $J_2$ into two parts:
\begin{equation*}
\begin{split}
J_2&\leq\bigg(\frac{1}{w^q(B)^{\kappa q/p}}\int_B\big|\xi(x)\big|^qw^q(x)\,dx\bigg)^{1/q}
+\bigg(\frac{1}{w^q(B)^{\kappa q/p}}\int_B\big|\zeta(x)\big|^qw^q(x)\,dx\bigg)^{1/q}\\
&:=J_3+J_4.
\end{split}
\end{equation*}
To deal with the term $J_3$, since $t=1+q/{p'}<q$, one has $w^q\in A^{\rho,\infty}_t\subset A^{\rho,\infty}_q.$
From the pointwise estimate \eqref{Tf2} and \eqref{wang3} in Lemma \ref{BMO1}, it then follows that
\begin{equation*}
\begin{split}
J_3&\leq C\cdot\frac{1}{w^q(B)^{\kappa/p}}\bigg(\int_B\big|b(x)-b_B\big|^qw^q(x)\,dx\bigg)^{1/q}\\
&\quad\times\sum_{k=1}^\infty\frac{1}{|2^{k+1}B|^{1-(\alpha/d)}}\int_{2^{k+1}B}\left(1+\frac{r}{\rho(x_0)}\right)^{N\cdot\frac{N_0}{N_0+1}}
\left(1+\frac{2^{k+1}r}{\rho(x_0)}\right)^{-N}|f(y)|\,dy\\
&\leq C_b\cdot\frac{w^q(B)^{1/q}}{w^q(B)^{\kappa/p}}\left(1+\frac{r}{\rho(x_0)}\right)^{\mu}\\
&\quad\times\sum_{k=1}^\infty\frac{1}{|2^{k+1}B|^{1-(\alpha/d)}}\int_{2^{k+1}B}\left(1+\frac{r}{\rho(x_0)}\right)^{N\cdot\frac{N_0}{N_0+1}}
\left(1+\frac{2^{k+1}r}{\rho(x_0)}\right)^{-N}|f(y)|\,dy.
\end{split}
\end{equation*}
Following along the same lines as that of Theorem $\ref{mainthm:1}$, we are able to show that
\begin{equation*}
\begin{split}
J_3&\leq C\big\|f\big\|_{L^{p,\kappa}_{\rho,\theta}(w^p,w^q)}
\left(1+\frac{r}{\rho(x_0)}\right)^{\mu}\left(1+\frac{r}{\rho(x_0)}\right)^{N\cdot\frac{N_0}{N_0+1}}\\
&\quad\times \sum_{k=1}^\infty\frac{w^q(B)^{1/q-\kappa/p}}{w^q(2^{k+1}B)^{1/q-\kappa/p}}
\left(1+\frac{2^{k+1}r}{\rho(x_0)}\right)^{-N+\theta+\theta'}\\
&\leq C\big\|f\big\|_{L^{p,\kappa}_{\rho,\theta}(w^p,w^q)}
\left(1+\frac{r}{\rho(x_0)}\right)^{\mu+N\cdot\frac{N_0}{N_0+1}}\\
&\quad\times\sum_{k=1}^\infty\left(\frac{|B|}{|2^{k+1}B|}\right)^{\delta{(1/q-\kappa/p)}}
\left(1+\frac{2^{k+1}r}{\rho(x_0)}\right)^{-N+\theta+\theta'+\eta{(1/q-\kappa/p)}}.
\end{split}
\end{equation*}
The last inequality is obtained by using \eqref{compare}. Next we estimate $\zeta(x)$ for any $x\in B(x_0,r)$. For any positive integer $N$, similar to the proof of \eqref{T} and \eqref{Tf2}, we can also deduce that
\begin{equation}\label{zeta}
\begin{split}
\zeta(x)
&=\int_{(2B)^c}\big|b(y)-b_B\big|\big|\mathcal K_{\alpha}(x,y)f(y)\big|\,dy\\
&\leq C_N\int_{(2B)^c}\big|b(y)-b_B\big|\bigg(1+\frac{|x-y|}{\rho(x)}\bigg)^{-N}\frac{1}{|x-y|^{d-\alpha}}\cdot|f(y)|\,dy\\
&\leq C_{N,d}\sum_{k=1}^\infty\int_{2^{k+1}B\backslash 2^k B}\big|b(y)-b_B\big|
\bigg(1+\frac{|x_0-y|}{\rho(x)}\bigg)^{-N}\frac{1}{|x_0-y|^{d-\alpha}}\cdot|f(y)|\,dy\\
&\leq C_{N,d}\sum_{k=1}^\infty\frac{1}{|2^{k+1}B|^{1-(\alpha/d)}}\int_{2^{k+1}B\backslash 2^k B}\big|b(y)-b_B\big|\bigg(1+\frac{2^kr}{\rho(x)}\bigg)^{-N}|f(y)|\,dy\\
&\leq C\sum_{k=1}^\infty\frac{1}{|2^{k+1}B|^{1-(\alpha/d)}}\int_{2^{k+1}B}\big|b(y)-b_B\big|\left(1+\frac{r}{\rho(x_0)}\right)^{N\cdot\frac{N_0}{N_0+1}}
\left(1+\frac{2^{k+1}r}{\rho(x_0)}\right)^{-N}|f(y)|\,dy,
\end{split}
\end{equation}
where in the last inequality we have used \eqref{com2} in Lemma \ref{N0}. Hence, by the above pointwise estimate for $\zeta(x)$,
\begin{equation}\label{inner}
\begin{split}
J_4&\leq C\cdot w^q(B)^{{(1/q-\kappa/p)}}\sum_{k=1}^\infty\left(1+\frac{r}{\rho(x_0)}\right)^{N\cdot\frac{N_0}{N_0+1}}
\left(1+\frac{2^{k+1}r}{\rho(x_0)}\right)^{-N}\\
&\times\frac{1}{|2^{k+1}B|^{1-(\alpha/d)}}\int_{2^{k+1}B}\big|b(y)-b_B\big|\big|f(y)\big|\,dy.
\end{split}
\end{equation}
Let us consider each term in the sum of \eqref{inner} separately. For each integer $k\geq 1$,
\begin{align}\label{est1}
&\frac{1}{|2^{k+1}B|^{1-(\alpha/d)}}\int_{2^{k+1}B}\big|b(y)-b_B\big|\big|f(y)\big|\,dy\notag\\
&\leq\frac{1}{|2^{k+1}B|^{1-(\alpha/d)}}\int_{2^{k+1}B}\big|b(y)-b_{2^{k+1}B}\big|\big|f(y)\big|\,dy\notag\\
&+\frac{1}{|2^{k+1}B|^{1-(\alpha/d)}}\int_{2^{k+1}B}\big|b_{2^{k+1}B}-b_B\big|\big|f(y)\big|\,dy.
\end{align}
By using H\"older's inequality, the first term of the expression \eqref{est1} is bounded by
\begin{equation*}
\begin{split}
&\frac{1}{|2^{k+1}B|^{1-(\alpha/d)}}\bigg(\int_{2^{k+1}B}\big|f(y)\big|^pw^p(y)\,dy\bigg)^{1/p}
\bigg(\int_{2^{k+1}B}\big|b(y)-b_{2^{k+1}B}\big|^{p'}w(y)^{-{p'}}\,dy\bigg)^{1/{p'}}\\
&\leq C\big\|f\big\|_{L^{p,\kappa}_{\rho,\theta}(w^p,w^q)}\cdot\frac{w^q(2^{k+1}B)^{{\kappa}/p}}{|2^{k+1}B|^{1-(\alpha/d)}}
\left(1+\frac{2^{k+1}r}{\rho(x_0)}\right)^{\theta}\bigg(\int_{2^{k+1}B}\big|b(y)-b_{2^{k+1}B}\big|^{p'}w(y)^{-{p'}}\,dy\bigg)^{1/{p'}}.
\end{split}
\end{equation*}
Since $w\in A^{\rho,\theta'}_{p,q}$ with $0<\theta'<\infty$ and $1<p<q<\infty$, then by Lemma \ref{Apq}, we know that $w^{-p'}\in A^{\rho,\theta'\cdot\frac{1}{1/q+1/{p'}}}_{t'}$ with $t'=1+{p'}/q$. If we denote $v=w^{-{p'}}$, then $v\in A^{\rho,\theta'\cdot\frac{1}{1/q+1/{p'}}}_{t'}\subset A^{\rho,\infty}_{p'}$ because $t'=1+{p'}/q<p'$. This fact together with Lemma \ref{BMO1} implies
\begin{equation*}
\begin{split}
&\bigg(\int_{2^{k+1}B}\big|b(y)-b_{2^{k+1}B}\big|^{p'}v(y)\,dy\bigg)^{1/{p'}}\\
&\leq C_{b}\cdot v\big(2^{k+1}B\big)^{1/{p'}}\left(1+\frac{2^{k+1}r}{\rho(x_0)}\right)^{\mu}\\
&=C_b\cdot\bigg(\int_{2^{k+1}B}w(y)^{-{p'}}\,dy\bigg)^{1/{p'}}\left(1+\frac{2^{k+1}r}{\rho(x_0)}\right)^{\mu}\\
&\leq C_{b,w}\cdot\frac{|2^{k+1}B|^{1-(\alpha/d)}}{w^q(2^{k+1}B)^{1/q}}\left(1+\frac{2^{k+1}r}{\rho(x_0)}\right)^{\theta'}
\left(1+\frac{2^{k+1}r}{\rho(x_0)}\right)^{\mu}.
\end{split}
\end{equation*}
Therefore, the first term of the expression \eqref{est1} can be bounded by a constant times
\begin{equation*}
\frac{w^q(2^{k+1}B)^{{\kappa}/p}}{w^q(2^{k+1}B)^{1/q}}\left(1+\frac{2^{k+1}r}{\rho(x_0)}\right)^{\theta+\theta'+\mu}.
\end{equation*}
Since $b\in \mathrm{BMO}_{\rho,\theta''}(\mathbb R^d)$ with $0<\theta''<\infty$, then by Lemma \ref{BMO2}, H\"older's inequality and the $A^{\rho,\theta'}_{p,q}$ condition on $w$, the latter term of the expression \eqref{est1} can be estimated by
\begin{equation*}
\begin{split}
&C_b(k+1)\left(1+\frac{2^{k+1}r}{\rho(x_0)}\right)^{\theta''}\frac{1}{|2^{k+1}B|^{1-(\alpha/d)}}\int_{2^{k+1}B}\big|f(y)\big|\,dy\\
\leq& C_b(k+1)\left(1+\frac{2^{k+1}r}{\rho(x_0)}\right)^{\theta''}\frac{1}{|2^{k+1}B|^{1-(\alpha/d)}}
\bigg(\int_{2^{k+1}B}\big|f(y)\big|^pw^p(y)\,dy\bigg)^{1/p}\bigg(\int_{2^{k+1}B}w(y)^{-{p'}}\,dy\bigg)^{1/{p'}}\\
\leq& C\big\|f\big\|_{L^{p,\kappa}_{\rho,\theta}(w^p,w^q)}\cdot(k+1)\frac{w^q(2^{k+1}B)^{{\kappa}/p}}{w^q(2^{k+1}B)^{1/q}}
\left(1+\frac{2^{k+1}r}{\rho(x_0)}\right)^{\theta+\theta'+\theta''}.
\end{split}
\end{equation*}
Consequently,
\begin{align}\label{est2}
&\frac{1}{|2^{k+1}B|^{1-(\alpha/d)}}\int_{2^{k+1}B}\big|b(y)-b_B\big|\big|f(y)\big|\,dy\notag\\
&\leq C\big\|f\big\|_{L^{p,\kappa}_{\rho,\theta}(w^p,w^q)}\cdot(k+1)\frac{w^q(2^{k+1}B)^{{\kappa}/p}}{w^q(2^{k+1}B)^{1/q}}
\left(1+\frac{2^{k+1}r}{\rho(x_0)}\right)^{\theta+\theta'+\theta''+\mu}.
\end{align}
Substituting the above inequality \eqref{est2} into \eqref{inner}, we thus obtain
\begin{equation*}
\begin{split}
J_4&\leq C\big\|f\big\|_{L^{p,\kappa}_{\rho,\theta}(w^p,w^q)}\cdot w^q(B)^{{(1/q-\kappa/p)}} \sum_{k=1}^\infty(k+1)\left(1+\frac{r}{\rho(x_0)}\right)^{N\cdot\frac{N_0}{N_0+1}}
\left(1+\frac{2^{k+1}r}{\rho(x_0)}\right)^{-N}\\
&\times\frac{1}{w^q(2^{k+1}B)^{{(1/q-\kappa/p)}}}\left(1+\frac{2^{k+1}r}{\rho(x_0)}\right)^{\theta+\theta'+\theta''+\mu}\\
&=C \big\|f\big\|_{L^{p,\kappa}_{\rho,\theta}(w^p,w^q)}
\left(1+\frac{r}{\rho(x_0)}\right)^{N\cdot\frac{N_0}{N_0+1}}\\
&\times\sum_{k=1}^\infty(k+1)\frac{w^q(B)^{{(1/q-\kappa/p)}}}{w^q(2^{k+1}B)^{{(1/q-\kappa/p)}}}
\left(1+\frac{2^{k+1}r}{\rho(x_0)}\right)^{-N+\theta+\theta'+\theta''+\mu}.
\end{split}
\end{equation*}
Note that $w^q\in A^{\rho,\theta'\cdot\frac{1}{1/q+1/{p'}}}_t$ with $t=1+q/{p'}$. A further application of \eqref{compare} yields
\begin{equation*}
\begin{split}
J_4&\leq C \big\|f\big\|_{L^{p,\kappa}_{\rho,\theta}(w^p,w^q)}
\left(1+\frac{r}{\rho(x_0)}\right)^{N\cdot\frac{N_0}{N_0+1}}\sum_{k=1}^\infty(k+1)\left(\frac{|B|}{|2^{k+1}B|}\right)^{\delta{(1/q-\kappa/p)}}\\
&\quad\times\left(1+\frac{2^{k+1}r}{\rho(x_0)}\right)^{\eta{(1/q-\kappa/p)}}\left(1+\frac{2^{k+1}r}{\rho(x_0)}\right)^{-N+\theta+\theta'+\theta''+\mu}\\
&=C \big\|f\big\|_{L^{p,\kappa}_{\rho,\theta}(w^p,w^q)}
\left(1+\frac{r}{\rho(x_0)}\right)^{N\cdot\frac{N_0}{N_0+1}}\sum_{k=1}^\infty(k+1)\left(\frac{|B|}{|2^{k+1}B|}\right)^{\delta{(1/q-\kappa/p)}}\\
&\quad\times\left(1+\frac{2^{k+1}r}{\rho(x_0)}\right)^{-N+\theta+\theta'+\theta''+\mu+\eta{(1/q-\kappa/p)}}.
\end{split}
\end{equation*}
Hence, combining the above estimates for $J_3$ and $J_4$, we conclude that
\begin{equation*}
\begin{split}
J_2\leq J_3+J_4&\leq C\big\|f\big\|_{L^{p,\kappa}_{\rho,\theta}(w^p,w^q)}
\left(1+\frac{r}{\rho(x_0)}\right)^{\mu+N\cdot\frac{N_0}{N_0+1}}\\
&\times\sum_{k=1}^\infty(k+1)\left(\frac{|B|}{|2^{k+1}B|}\right)^{\delta{(1/q-\kappa/p)}}
\left(1+\frac{2^{k+1}r}{\rho(x_0)}\right)^{-N+\theta+\theta'+\theta''+\mu+\eta{(1/q-\kappa/p)}}.
\end{split}
\end{equation*}
By choosing $N$ large enough so that $N>\theta+\theta'+\theta''+\mu+\eta{(1/q-\kappa/p)}$, we thus have
\begin{equation*}
\begin{split}
J_2&\leq C\left(1+\frac{r}{\rho(x_0)}\right)^{\mu+N\cdot\frac{N_0}{N_0+1}}
\sum_{k=1}^\infty(k+1)\left(\frac{|B|}{|2^{k+1}B|}\right)^{\delta{(1/q-\kappa/p)}}\\
&\leq C\left(1+\frac{r}{\rho(x_0)}\right)^{\mu+N\cdot\frac{N_0}{N_0+1}},
\end{split}
\end{equation*}
where the last series is convergent since $0<\kappa<p/q$. Finally, collecting the above estimates for $J_1,J_2$, and letting $\vartheta=\max\big\{\vartheta',\mu+N\cdot\frac{N_0}{N_0+1}\big\}$, we obtain the desired inequality \eqref{Main3}. The proof of Theorem \ref{mainthm:3} is finished.
\end{proof}

The higher order commutators generated by $\mathrm{BMO}_{\rho,\infty}(\mathbb R^d)$ functions $b$ and the fractional integrals $\mathcal I_{\alpha}$ are usually defined by
\begin{equation*}
\begin{cases}
[b,\mathcal I_{\alpha}]_mf(x):=\displaystyle\int_{\mathbb R^n}\big[b(x)-b(y)\big]^m \mathcal K_{\alpha}(x,y)f(y)\,dy,\quad x\in\mathbb R^d;&\\
\quad 0<\alpha<d,~~~m=1,2,3,\dots.&
\end{cases}
\end{equation*}
Obviously, $[b,\mathcal I_{\alpha}]_1=[b,\mathcal I_{\alpha}]$ which is just the linear commutator \eqref{briesz}, and
\begin{equation*}
[b,\mathcal I_{\alpha}]_m=\big[b,[b,\mathcal I_{\alpha}]_{m-1}\big],\quad m=2,3,\dots.
\end{equation*}
By induction on $m$, we are able to show that the conclusion of Theorem \ref{mainthm:3} also holds for the higher order commutators $[b,\mathcal I_{\alpha}]_m$ with $m\geq2$. The details are omitted here.

\begin{thm}
Let $0<\alpha<d$, $1<p<d/{\alpha}$, $1/q=1/p-{\alpha}/d$ and $w\in A^{\rho,\infty}_{p,q}$. If $V\in RH_s$ with $s\geq d/2$ and $0<\kappa<p/q$, then for any positive integer $m\geq2$, the higher order commutators $[b,\mathcal I_{\alpha}]_m$ are bounded from $L^{p,\kappa}_{\rho,\infty}(w^p,w^q)$ into $L^{q,{(\kappa q)}/p}(w^q)$, whenever $b\in\mathrm{BMO}_{\rho,\infty}(\mathbb R^d)$.
\end{thm}

\section{Proof of Theorem \ref{mainthm:5}}

The following lemma plays a key role in the proof of our main theorem, which can be found in \cite[Proposition 8]{bong5} (see also \cite[Lemma 3.7]{tang}).

\begin{lem}[\cite{bong5}]\label{kernel2}
Let $V\in RH_s$ with $s\geq d/2$ and $0<\alpha<d$. For every positive integer $N\geq1$, there exists a positive constant $C_N>0$ such that for all $x$ and $y$ in $\mathbb R^d$, and for some fixed $0<\varepsilon\leq 1$,
\begin{equation*}
\big|\mathcal K_{\alpha}(x,z)-\mathcal K_{\alpha}(y,z)\big|\leq C_N\bigg(1+\frac{|x-z|}{\rho(x)}\bigg)^{-N}\frac{|x-y|^{\varepsilon}}{|x-z|^{d-\alpha+\varepsilon}},
\end{equation*}
whenever $|x-y|\leq |x-z|/2$.
\end{lem}

\begin{proof}[Proof of Theorem $\ref{mainthm:5}$]
For an arbitrary ball $B=B(x_0,r)$ in $\mathbb R^d$ and $0<\alpha<d$, it suffices to prove that the following inequality
\begin{equation}\label{end1.1}
\frac{1}{|B|}\int_B\big|\mathcal I_{\alpha}f(x)-(\mathcal I_{\alpha}f)_B\big|\,dx\leq C\cdot\left(1+\frac{r}{\rho(x_0)}\right)^{\vartheta}
\end{equation}
holds for any $f\in L^{p,\kappa}_{\rho,\infty}(w^p,w^q)$ with $1<p<q<\infty$ and $\kappa=p/q$, where $(\mathcal I_{\alpha}f)_B$ denotes the average of $\mathcal I_{\alpha}f$ over $B$. Suppose that $f\in L^{p,\kappa}_{\rho,\theta}(w^p,w^q)$ for some $\theta>0$ and $w\in A^{\rho,\theta'}_{p,q}$ for some $\theta'>0$. Decompose the function $f$ as $f=f_1+f_2$, where $f_1=f\cdot\chi_{4B}$, $f_2=f\cdot\chi_{(4B)^c}$, $4B=B(x_0,4r)$. By the linearity of the $\mathcal L$-fractional integral operator $\mathcal I_{\alpha}$, the left-hand side of \eqref{end1.1} can be divided into two parts. That is,
\begin{equation*}
\begin{split}
&\frac{1}{|B|}\int_B\big|\mathcal I_{\alpha}f(x)-(\mathcal I_{\alpha}f)_B\big|\,dx\\
&\leq\frac{1}{|B|}\int_B\big|\mathcal I_{\alpha}f_1(x)-(\mathcal I_{\alpha}f_1)_B\big|\,dx
+\frac{1}{|B|}\int_B\big|\mathcal I_{\alpha}f_2(x)-(\mathcal I_{\alpha}f_2)_B\big|\,dx\\
&:=K_1+K_2.
\end{split}
\end{equation*}
First let us consider the term $K_1$. Applying the weighted $(L^p,L^q)$-boundedness of $\mathcal I_{\alpha}$ (see Theorem \ref{strong}) and H\"older's inequality, we obtain
\begin{equation*}
\begin{split}
K_1&\leq\frac{2}{|B|}\int_B|\mathcal I_{\alpha}f_1(x)|\,dx\\
&\leq\frac{2}{|B|}\bigg(\int_B|\mathcal I_{\alpha}f_1(x)|^qw^q(x)\,dx\bigg)^{1/q}\bigg(\int_B w(x)^{-q'}dx\bigg)^{1/{q'}}\\
&\leq\frac{C}{|B|}\bigg(\int_{4B}|f(x)|^pw^p(x)\,dx\bigg)^{1/p}\bigg(\int_B w(x)^{-q'}dx\bigg)^{1/{q'}}\\
&\leq C\big\|f\big\|_{L^{p,\kappa}_{\rho,\theta}(w^p,w^q)}
\cdot\frac{w^q(4B)^{{\kappa}/p}}{|B|}\bigg(\int_B w(x)^{-q'}dx\bigg)^{1/{q'}}\left(1+\frac{4r}{\rho(x_0)}\right)^{\theta}.
\end{split}
\end{equation*}
Since $q'<p'$, by H\"older's inequality, it is easy to check that
\begin{equation*}
\bigg(\frac{1}{|B|}\int_B w(x)^{-q'}\,dx\bigg)^{1/{q'}}\leq\bigg(\frac{1}{|B|}\int_B w(x)^{-p'}dx\bigg)^{1/{p'}}.
\end{equation*}
Moreover, since $w$ is a weight in the class $A^{\rho,\theta'}_{p,q}$, one has
\begin{equation*}
\begin{split}
&\left(\frac1{|B|}\int_B w(x)^q\,dx\right)^{1/q}\bigg(\frac{1}{|B|}\int_B w(x)^{-q'}dx\bigg)^{1/{q'}}\\
&\leq\bigg(\frac{1}{|B|}\int_B w(x)^q\,dx\bigg)^{1/q}\bigg(\frac{1}{|B|}\int_B w(x)^{-{p'}}\,dx\bigg)^{1/{p'}}
\leq C\cdot\left(1+\frac{r}{\rho(x_0)}\right)^{\theta'},
\end{split}
\end{equation*}
which implies
\begin{equation}\label{end1.2}
\bigg(\int_B w(x)^{-q'}dx\bigg)^{1/{q'}}\leq C\cdot\left(1+\frac{r}{\rho(x_0)}\right)^{\theta'}\frac{|B|}{w^q(B)^{1/q}}.
\end{equation}
Also observe that $w^q\in A^{\rho,\theta'\cdot\frac{1}{1/q+1/{p'}}}_t$ with $t=1+q/{p'}$. Using the inequalities \eqref{doubling1} and \eqref{end1.2}, and noting the fact that $\kappa=p/q$, we have
\begin{equation*}
\begin{split}
K_1&\leq C\big\|f\big\|_{L^{p,\kappa}_{\rho,\theta}(w^p,w^q)}\cdot\frac{w^q(4B)^{1/q}}{w^q(B)^{1/q}}
\left(1+\frac{4r}{\rho(x_0)}\right)^{\theta}\left(1+\frac{r}{\rho(x_0)}\right)^{\theta'}\\
&\leq C\big\|f\big\|_{L^{p,\kappa}_{\rho,\theta}(w^p,w^q)}\left(1+\frac{4r}{\rho(x_0)}\right)^{2\theta'}
\left(1+\frac{4r}{\rho(x_0)}\right)^{\theta}\left(1+\frac{r}{\rho(x_0)}\right)^{\theta'}\\
&\leq C\cdot\left(1+\frac{r}{\rho(x_0)}\right)^{\vartheta'},
\end{split}
\end{equation*}
where $\vartheta':={3\theta'}+\theta$. Now we turn to estimate $K_2$. For any $x\in B(x_0,r)$,
\begin{equation*}
\begin{split}
\big|\mathcal I_{\alpha}f_2(x)-(\mathcal I_{\alpha}f_2)_B\big|
&=\bigg|\frac{1}{|B|}\int_B\big[\mathcal I_{\alpha}f_2(x)-\mathcal I_{\alpha}f_2(y)\big]\,dy\bigg|\\
&=\bigg|\frac{1}{|B|}\int_B\bigg\{\int_{(4B)^c}\Big[\mathcal K_{\alpha}(x,z)-\mathcal K_{\alpha}(y,z)\Big]f(z)\,dz\bigg\}dy\bigg|\\
&\leq\frac{1}{|B|}\int_B\bigg\{\int_{(4B)^c}\big|\mathcal K_{\alpha}(x,z)-\mathcal K_{\alpha}(y,z)\big|\cdot|f(z)|\,dz\bigg\}dy.
\end{split}
\end{equation*}
Next note that by a purely geometric argument, one has
\begin{equation*}
|x-y|\leq |x-z|/2 \quad \& \quad |x-z|\approx |x_0-z|,
\end{equation*}
whenever $x,y\in B$ and $z\in(4B)^c$. This fact along with Lemma \ref{kernel2} yields
\begin{align}\label{average}
\big|\mathcal I_{\alpha}f_2(x)-(\mathcal I_{\alpha}f_2)_B\big|
&\leq\frac{C_N}{|B|}\int_B\bigg\{\int_{(4B)^c}\bigg(1+\frac{|x-z|}{\rho(x)}\bigg)^{-N}
\frac{|x-y|^{\varepsilon}}{|x-z|^{d-\alpha+\varepsilon}}\cdot|f(z)|\,dz\bigg\}dy\notag\\
&\leq C_{N,d}\int_{(4B)^c}\bigg(1+\frac{|x_0-z|}{\rho(x)}\bigg)^{-N}\frac{r^{\varepsilon}}{|x_0-z|^{d-\alpha+\varepsilon}}\cdot|f(z)|\,dz\notag\\
&=C_{N,d}\sum_{k=2}^\infty\int_{2^{k+1}B\backslash 2^k B}
\bigg(1+\frac{|x_0-z|}{\rho(x)}\bigg)^{-N}\frac{r^{\varepsilon}}{|x_0-z|^{d-\alpha+\varepsilon}}\cdot|f(z)|\,dz\notag\\
&\leq C_{N,d}\sum_{k=2}^\infty\frac{1}{2^{k\varepsilon}}\cdot\frac{1}{|2^{k+1}B|^{1-({\alpha}/d)}}
\int_{2^{k+1}B\backslash 2^k B}\bigg(1+\frac{2^kr}{\rho(x)}\bigg)^{-N}|f(z)|\,dz.
\end{align}
Furthermore, by using H\"older's inequality, \eqref{com2} and $A^{\rho,\theta'}_{p,q}$ condition on $w$, we get that for any $x\in B(x_0,r)$,
\begin{align}\label{end1.3}
\big|\mathcal I_{\alpha}f_2(x)-(\mathcal I_{\alpha}f_2)_B\big|
&\leq C\sum_{k=2}^\infty\frac{1}{2^{k\varepsilon}}\cdot\frac{1}{|2^{k+1}B|^{1-({\alpha}/d)}}
\left(1+\frac{r}{\rho(x_0)}\right)^{N\cdot\frac{N_0}{N_0+1}}
\left(1+\frac{2^{k+1}r}{\rho(x_0)}\right)^{-N}\notag\\
&\times\bigg(\int_{2^{k+1}B}\big|f(z)\big|^pw^p(z)\,dz\bigg)^{1/p}
\left(\int_{2^{k+1}B}w(z)^{-p'}dz\right)^{1/{p'}}\notag\\
&\leq C\big\|f\big\|_{L^{p,\kappa}_{\rho,\theta}(w^p,w^q)}
\sum_{k=2}^\infty\frac{1}{2^{k\varepsilon}}\cdot\left(1+\frac{r}{\rho(x_0)}\right)^{N\cdot\frac{N_0}{N_0+1}}
\left(1+\frac{2^{k+1}r}{\rho(x_0)}\right)^{-N}\notag\\
&\times\frac{w^q(2^{k+1}B)^{{\kappa}/p}}{w^q(2^{k+1}B)^{1/q}}\left(1+\frac{2^{k+1}r}{\rho(x_0)}\right)^{\theta}
\left(1+\frac{2^{k+1}r}{\rho(x_0)}\right)^{\theta'}\notag\\
&=C\big\|f\big\|_{L^{p,\kappa}_{\rho,\theta}(w^p,w^q)}
\sum_{k=2}^\infty\frac{1}{2^{k\varepsilon}}\cdot\left(1+\frac{r}{\rho(x_0)}\right)^{N\cdot\frac{N_0}{N_0+1}}
\left(1+\frac{2^{k+1}r}{\rho(x_0)}\right)^{-N+\theta+\theta'},
\end{align}
where the last equality is due to the assumption $\kappa=p/q$. From the pointwise estimate \eqref{end1.3}, it readily follows that
\begin{equation*}
\begin{split}
K_2&=\frac{1}{|B|}\int_B\big|\mathcal I_{\alpha}f_2(x)-(\mathcal I_{\alpha}f_2)_B\big|\,dx\\
&\leq C\big\|f\big\|_{L^{p,\kappa}_{\rho,\theta}(w^p,w^q)}
\sum_{k=2}^\infty\frac{1}{2^{k\varepsilon}}\cdot\left(1+\frac{r}{\rho(x_0)}\right)^{N\cdot\frac{N_0}{N_0+1}}
\left(1+\frac{2^{k+1}r}{\rho(x_0)}\right)^{-N+\theta+\theta'}.
\end{split}
\end{equation*}
Now $N$ can be chosen sufficiently large so that $N>\theta+\theta'$, and hence the above series is convergent. Therefore,
\begin{equation*}
K_2\leq C\left(1+\frac{r}{\rho(x_0)}\right)^{N\cdot\frac{N_0}{N_0+1}}\sum_{k=2}^\infty\frac{1}{2^{k\varepsilon}}\leq C\left(1+\frac{r}{\rho(x_0)}\right)^{N\cdot\frac{N_0}{N_0+1}}.
\end{equation*}
Fix this $N$ and set $\vartheta=\max\big\{\vartheta',N\cdot\frac{N_0}{N_0+1}\big\}$. Finally, combining the above estimates for $K_1$ and $K_2$, the inequality \eqref{end1.1} is proved and then the proof of Theorem \ref{mainthm:5} is finished.
\end{proof}

\end{document}